\documentclass[11pt,english,reqno]{amsart}
\usepackage{bm}
\usepackage{color}
\usepackage{xcolor}
\usepackage{babel}
\usepackage{verbatim}
\usepackage{amsthm}
\usepackage{amstext}
\usepackage{amssymb}
\usepackage{amsbsy}
\usepackage{ifpdf}
\usepackage[nobysame,abbrev,alphabetic]{amsrefs}
\usepackage{enumerate}
\usepackage{amsmath}
\usepackage{amscd}
\usepackage{amsfonts}
\usepackage{graphicx}
\usepackage[all]{xy}
\usepackage{verbatim}
\usepackage{gensymb}
\usepackage{mathrsfs}
\usepackage[colorlinks]{hyperref}
\allowdisplaybreaks
\hypersetup{
linkcolor=blue,          
citecolor=blue,        
urlcolor=blue
}
\usepackage[top=1in,bottom=1in,left=1in,right=1in]{geometry}
\newtheorem{theorem}{Theorem}[section]

\newtheorem{lemma}[theorem]{Lemma}
\newtheorem{corollary}[theorem]{Corollary}
\theoremstyle{definition}

\theoremstyle{theorem}\newtheorem{proposition}[theorem]{Proposition}

\newtheorem{problem}[theorem]{Problem}
\theoremstyle{definition}

\theoremstyle{definition}
\theoremstyle{definition}\newtheorem{remark}[theorem]{Remark}
\theoremstyle{definition}
\newcommand{\al}{\alpha}
\newcommand{\be}{\beta}

\newcommand{\Ga}{\Gamma}
\newcommand{\del}{\delta}

\newcommand{\Lam}{\Lambda}

\newcommand{\sig}{\sigma}

\newcommand{\vphi}{\varphi}

\newcommand{\cP}{\mathcal{P}}

\newcommand{\bR}{\mathbb{R}}
\newcommand{\bZ}{\mathbb{Z}}
\newcommand{\bQ}{\mathbb{Q}}

\newcommand{\bK}{\mathbb{K}}

\newcommand{\bS}{\mathbb{S}}

\newcommand{\SO}{\operatorname{SO}}

\newcommand{\defi}{\overset{\on{def}}{=}}

\newcommand\norm[1]{||#1||}

\newcommand\set[1]{\left\{#1\right\}}
\newcommand\pa[1]{\left(#1\right)}

\newcommand\av[1]{|#1|}
\newcommand\on[1]{\operatorname{#1}}

\newcommand\mb[1]{\mathbf{#1}}

\newcommand\tb[1]{\textbf{#1}}
\newcommand\mat[1]{\pa{\begin{matrix}#1\end{matrix}}}

\newcommand\smallmat[1]{\pa{\begin{smallmatrix}#1\end{smallmatrix}}}

\newcommand{\supp}{\on{supp}}

\newcommand{\SL}{\operatorname{\mathbf{SL}}}
\newcommand{\GL}{\operatorname{\mathbf{GL}}}
\newcommand{\PGL}{\operatorname{\mathbf{PGL}}}
\newcommand{\Lie}{\operatorname{Lie}}
\newcommand{\Ad}{\operatorname{Ad}}

\newcommand{\diag}{\operatorname{diag}}

\newcommand{\cov}{\operatorname{cov}}

\newcommand{\nubest}{\nu_{\tb{best}}}
\newcommand{\nuvbest}{\nu_{\tb{best}}^{\vec{\al}}}

\newcommand{\lra}{\longrightarrow}

\newcommand{\onto}{\xymatrix{\ar@{>>}[r]&}}
\newcommand{\slra}[1]{\stackrel{#1}{\lra}}

\begin{document}

\title[Translates of orbits and applications]{Translates of S-arithmetic periodic orbits and applications}


%
\author{Uri Shapira and Cheng Zheng}
\begin{abstract}
We prove that certain sequences of periodic orbits of the diagonal group in the space of lattices equidistribute. As an application we obtain
new information regarding the sequence of best approximations to certain vectors with algebraic coordinates. In order to prove these results
we generalize the seminal work of Eskin Mozes and Shah about the equidistribution of translates of periodic measures from the real 
case to the $S$-arithmetic case. 
\end{abstract}
\address{U.S.\ Department of Mathematics\
Technion \\
Haifa \\
Israel}
\email{ushapira@technion.ac.il}
\address{
 C.Z.\ School of Mathematical Sciences Shanghai Jiao Tong University \\
 Shanghai\\
 China}
 \email{zheng.c@sjtu.edu.cn}
\thanks{The authors acknowledge the support of ISF grants number 662/15 and 871/17, and the second author acknowledges the support at Technion by a Fine Fellowship. This work has received funding from the European Research Council (ERC) under the European Union's Horizon 2020 research and innovation programme (Grant Agreement No. 754475).}
\maketitle

\section{Introduction}\label{intr}
\subsection{Context}
Let $G$ be a locally compact second countable topological group and $\Ga<G$ a closed subgroup (usually assumed to be
discrete or even a lattice in $G$). 
The space $X = \Ga\backslash G$ endowed with
the quotient topology is then called a homogeneous space. The group $G$ and its various subgroups act on
$X$ from the left by $h\cdot \Ga g = \Ga gh^{-1}$.  
In homogeneous dynamics one studies these actions and tries to understand orbits. Given a closed subgroup $H<G$
the most fundamental question is to understand topological properties of $H$-orbits: Which orbits are closed, dense or bounded? What can be said about orbit closures in general? A fundamental role in the discussion is played by periodic 
probability measures. Let $\cP(X)$ denote the space of Borel probability measures on $X$ equipped with 
the weak$^*$ topology and note that the left action of $G$ on $X$ induces a continuous left action on $\cP(X)$ by 
$(g,\mu)\mapsto g\mu$ where for $f\in C_c(X)$, $\int_X f(x) d(g\mu) = \int_X f(g^{-1}x) d\mu$. 
An orbit $Hx\subset X$ is called periodic if there exists $\mu\in\cP(X)$ which is $H$-invariant (i.e. $h\mu=\mu$ for all
$h\in H$) and such that $\mu(Hx) =1$. The measure $\mu$ is unique in this case
and will be referred to
as the algebraic or periodic measure on the orbit.
It is denoted by 
$\mu_{Hx}$ or $\mu_{\Ga\backslash \Ga g H}$ if $x= \Ga g$. 

A significant portion of the literature in homogeneous dynamics is devoted to the following:
\begin{problem}\label{mainprob}
Understand the weak$^*$-accumulation points (limits)
 of periodic measures. That is, what can be said about possible 
limits of sequences $\mu_{H_n x_n}$ where $H_nx_n$ are periodic orbits. 
\end{problem}
\noindent Common phenomena related to Problem~\ref{mainprob} in particular cases 
are:
\begin{enumerate}
\item There are limit measures in $\cP(X)$  -- this will be referred to as non-escape of mass. 
\item All limit measures are periodic themselves -- this will be referred to as rigidity.
\item Often, when non-escape of mass and rigidity occurs, if one imposes some 
mild conditions on the sequence $H_nx_n$ (usually of algebraic nature), then
all limit measures must be $G$-invariant. Since there is a unique $G$-invariant measure on $X$, this actually implies that the sequence converges to it -- a phenomenon which we refer to as equidistribution.
\end{enumerate}
\noindent In the general setting described above Problem \ref{mainprob} might be too hard to tackle but under various technical assumptions
on the players $(G,\Ga, H_n, x_n)$ quite a bit is known. A standing assumption will be that $\Ga<G$ is a lattice. We mention
two seminal works in this direction:
\begin{itemize}
\item In \cite{MS95} Mozes and Shah solve this problem when $G, H_n$ are real Lie groups and the periodic 
measure $\mu_{H_n x_n}$ is ergodic under the action of a one-parameter unipotent subgroup of $H_n$. It 
is paramount for their consideration that $H_n$ will contain unipotent elements and thus when $H_n$ is a reductive
group - and in particular, a maximal torus - 
their results do not apply. 
 \item When $G, H$ are reductive algebraic groups defined over $\bQ$, $\Ga<G$ is 
an arithmetic lattice in $G$ and the centralizer of $H$ does not have any non-trivial $\bQ$-characters, 
then in the seminal papers \cite{EMS96, EMS97} Eskin, Mozes, and Shah analyze the possible
limit points of $\mu_{g_n Hx}$ of deformations of a fixed periodic $H$-orbit (note that $g_nHx = g_nH g_n^{-1} g_nx$ so this falls into
the above setting). 
\end{itemize}
In applications, it is often required to extend the above results to the $S$-arithmetic setting. The situation in the literature regarding \cite{MS95} is satisfactory. In \cite{GO11} Gorodnik and Oh generalize \cite{MS95} to the $S$-arithmetic setting.
In this paper, we are motivated by an application which require an $S$-arithmetic version of \cite{EMS96, EMS97}.
As far as we are aware of, there
is no published paper in the literature establishing such a result and so in order to pursue our application 
we take upon the endeavour to produce such a proof in a setting general enough for our application. We note that 
there is a paper on arxiv \cite{RZ16} in which an $S$-arithmetic generalization of \cite{EMS96, EMS97} is given.
The formalism there is quite sophisticated and we were not able to understand how to apply
the results from that paper nor to verify its correctness. Eventually we preferred to give a self contained treatment 
even on the expense of compromising brevity. 
\if
We restrict attention to the case where $H$ is a maximal $\bR$-split torus in 
a semisimple $\bR$-split group $G$ (see \S\ref{sec:ldto}). Proving an $S$-arithmetic version of \cite{EMS96, EMS97} for more general $G,H$'s 
seem to be more challenging and require quite a bit more technicalities.
\fi
\subsection{Motivating example}\label{sec:main-example}
Let $X = \Ga\backslash G$ be the space of homothety classes of lattices in $\bR^d$. Here $G=\PGL_d(\bR)$, $\Ga = \PGL_d(\bZ)$. An element $g\in G$ will be represented by a matrix in $\GL_d(\bR)$ which by abuse of notation we 
also denote by $g$ and the coset $\Ga g$ represents the homothety class of the lattice spanned by the rows of\footnote{The alert reader will notice that this identification is not equivariant with respect to the left $G$-action on $X$ and the natural left action on the space of lattices. Under this identification if a coset $x = \Ga g$ represents a lattice $\Lam\subset \bR^d$, then $gx$ represents the lattice $g^*\Lam$, where $g^* = (g^{-1})^{\on{t}}$.} $g$, $g^{\on{t}}\bZ^d$. 
Let $A<G$ denote the (connected component of the identity of the) group of diagonal matrices. 

Let $\bK/\bQ$ be a totally real number field of degree $d$. Let $\sig_j:\bK\to \bR$ ($j=1,\dots,d$) be the distinct embeddings of $\bK$ into $\bR$
ordered in some way. 
Let $\vphi:\bK\to \bR^d$ be the map $\vphi = (\sig_1,\dots,\sig_d)$. 
For an ordered basis $\vec{\al} = (\al_1,\dots,\al_d)$ of $\bK/\bQ$ let 
\begin{equation}\label{eq:2301}
g_{\vec{\al} }=
\pa{\begin{smallmatrix}
-&\vphi(\al_1)&-\\ 
& \vdots & \\
-&\vphi(\al_d)&-
\end{smallmatrix}}
\end{equation}
and let $x_{\vec{\al}} = \Ga g_{\vec{\al}}$. Then, it is well known (see e.g.\ \cite{LW01}) that the orbit $Ax_{\vec{\al}}$ is periodic. The motivating 
problem which started
this paper was to establish the following result which is proved in
\S\ref{sec:proofs-of-apps}.
\begin{theorem}\label{th110}
Let $\bK/\bQ$ be a degree $d$ totally real number field with basis $\vec{\al} = (\al_1,\dots,\al_d)$. For any $m\in \bZ$ ($m\ne 0,\pm1$) and 
for any sequence $\mb{i}_n = (i_{n,1},\dots,i_{n,d})\in \bZ^d$ consider
the sequence 
$\vec{\al}_n \defi (m^{i_{n,1}}\al_1, m^{i_{n,2}}\al_2,\dots, m^{i_{n,d}}\al_d)$. Then, for any choice of elements $s_n\in G$ the sequence of periodic orbits $s_nAx_{\vec{\al}_n} = \Ga g_{\vec{\al}_n}As_n^{-1}$
equidistribute in $X$ as long for any $1\le j<r\le d$, $\lim_n \av{i_{n,j}-i_{n,r}} = \infty.$
\end{theorem}
\begin{remark}
Theorem \ref{th110} attains an appealing form when one chooses the basis
$1,\be,\dots, \be^{d-1}$ for $\be\in \bK$ which generates the extension. In this case we write $x_\be$ for the corresponding lattice.
Given a generator $\al\in \bK$, we may choose in Theorem~\ref{th110} the sequence $\mb{i}_n = (0, n,2n,\dots, (d-1)n)$ 
and obtain that the sequence of periodic orbits $Ax_{m^n\al}$ equidistribute in $X$.
\end{remark}

Theorem~\ref{th110} falls into the setting of Problem~\ref{mainprob}. The sequence of orbits under consideration
is a deformation of a fixed orbit but this deformation is done in an $S$-arithmetic
extension of $X$ and then projected back to $X$. This is why the results of \cite{EMS96,EMS97} do not apply and 
one needs an $S$-arithmetic version of them. 
In the case $d=2$ (with $s_n=e$) this equidistribution was established in \cite{AS} using different techniques. There the authors 
used mixing (and hence obtain effective equidistribution). For $d\ge 3$ the strategy of \cite{AS} fails and one needs to appeal to 
the techniques of \cite{EMS96, EMS97} as we do here.

In \cite{AS} this equidistribution result was applied to deduce statistical information about the periods in the continued 
fraction expansions of quadratic numbers of the form $m^n\al$ as $n\to\infty$. In \S\ref{application} we apply Theorem~\ref{th110}
to obtain new results regarding the statistics of best approximations of vectors like $(m\al, m^2\al^2,\dots,m^{d-1}\al^{d-1})$ where $\bQ(\al)$
is a totally real number field of degree $d$ over $\bQ$.

\subsection{An application: Best approximations of algebraic vectors}\label{application}
In this subsection we wish to demonstrate the significance of Theorem \ref{th110} to the theory of Diophantine approximaiton. Readers whose 
interest lies solely in the dynamical aspects can skip this subsection altogether. 
We briefly recall some basic concepts from the theory of Diophantine approximation. Given an irrational vector $v\in \bR^{d-1}$ there is a natural correspondence between rational vectors $\frac{1}{q}\tb{p}$ (here $q\in\bZ, \tb{p}\in\bZ^{d-1}$ and $\gcd(\tb{p},q) =1$) close to $v$ and primitive integral 
vectors $\smallmat{\tb{p} \\ q}\in\bZ^d$ approximating the line $\bR\cdot \smallmat{v\\1}$. One defines 
the sequence of best approximations to $v$ to be a sequence of primitive integral vectors $\smallmat{\tb{p}_k\\q_k}\in \bZ^d$ 
defined recursively as follows:  $q_1 = 1$ and $\tb{p}_1$  
is (one of the potentially finitely many) integral vectors
satisfying $\norm{\tb{p}_1 -v} = \min_{\tb{p}\in\bZ^{d-1}}\norm{\tb{p}-v}$. For $k >1$, $q_k$ is the minimal integer $q$
for which 
$$ \min\set{\norm{q v -\tb{p}}: \tb{p}\in\bZ^{d-1}} < \norm{q_{k-1}v - \tb{p}_{k-1}}$$
and $\tb{p}_k$ is then chosen so that 
$$\norm{q_kv -\tb{p}_k} =\min\set{\norm{q_k v -\tb{p}}: \tb{p}\in\bZ^{d-1}}.$$
A more geometric way to define (and understand) the vectors $\smallmat{\mb{p}_k\\q_k}$ is as follows: One travels along the ray 
$\set{t\cdot\smallmat{v\\1}:t\ge 1}$ and records the closest integral vectors one sees along the way; that is, one records an integral vector
at a certain (integral) time if its distance to the ray beats the distances recorded thus far. In this description the distance one uses
on $\bR^d$ is not the euclidean one but $\on{d}(\smallmat{w_1\\r_1}, \smallmat{w_2\\r_2}) = \max\set{\norm{w_1-w_2},\av{r_1-r_2}}$ where
$w_i\in\bR^{d-1}$ and $r_i\in\bR$ and the norm on $\bR^{d-1}$ is the euclidean norm (in fact, one can choose any norm on $\bR^{d-1}$ which is 
of interest).  This way, each irrational vector $v\in\bR^{d-1}$ yields a sequence of primitive integral vectors $\smallmat{\mb{p}_k\\q_k}$ which 
is referred to as the sequence of best approximations to the ray $\bR\cdot \smallmat{v\\1}$ or by slight abuse of language, to $v$.

The sequence of best approximations is a central object of study in Diophantine approximation. Recently, in \cite{SWbest}, 
results pertaining to statistical information regarding these sequences in two cases were established: For Lebesgue almost any $v$
some information about the behaviour of $\smallmat{\mb{p}_k\\q_k}$ was established and it was shown that they obey certain 
universal laws. In addition, for certain vectors $v$ with algebraic coordinates it was shown that the sequence 
$\smallmat{\mb{p}_k\\q_k}$ obeys statistical laws which are different from the universal ones. As an application of Theorem~\ref{th110} we
show that if $\bQ(\al)$ is a totally real field of degree $d$ over $\bQ$ and if $m\ne 0,\pm{1}$ is an integer then 
the statistical laws corresponding to $v_n = (m^n\al,m^{2n}\al^2,\dots,m^{(d-1)n}\al^{d-1})$ approach the universal statistical law
as $n\to\infty$. To be more precise we need to define the objects pertaining the random 
variables which obey these statistical laws. 

To each best approximation vector $\smallmat{\tb{p}_k\\q_k}$ we associate a triple 
$$(\Lam_k, w_k, \smallmat{\tb{p}_k\\q_k})\in X_{d-1}\times \bR^{d-1}\times \hat{\bZ}^d$$ 
where $X_{d-1}$ is the space of lattices of covolume 1 in $\bR^{d-1}$, 
and $\hat{\bZ} = \prod_p \bZ_p$ is the pro-finite completion of $\bZ$. 
This correspondence is defined in the following way: The lattice $\Lam_k$ is referred to as the \textit{directional lattice}
of $\bZ^d$ in direction of the $k$'th best approximation $\smallmat{\tb{p}_k\\q_k}$. It is defined by
\begin{align*}
\Lam_k & \defi q_k^{1/(d-1)}\pi^{\smallmat{\tb{p}_k\\q_k}}_{\bR^{d-1}}(\bZ^n),
\end{align*}
where $ \pi^{w}_{\bR^{d-1}}$ is the linear projection of $\bR^d$ onto the hyperplane $\set{\smallmat{u\\ 0}: u\in\bR^{d-1}}$ with kernel equal to $\bR\cdot w$. Thus, $\Lam_k$ is the unimodular lattice of the horizontal hyperplane $\bR^{d-1}\times\set{0}$ which records how $\bZ^d$ looks like from the direction of the $k$'th best approximation to $v$. 

The vector $w_k\in\bR^{d-1}$ is referred to as the \textit{displacement vector} of the $k$'th best approximation $\smallmat{\tb{p}_k\\q_k}$.
It is defined by
\begin{align*}
w_k &\defi q_k^{-1/(d-1)}(q_kv-\mb{p}_k).
\end{align*}
This vector tells us two things: First, its direction $w_k/\norm{w_k}$ on the unit sphere $\bS^{d-2}$ is the 
direction from which the $k$'th best approximation approaches the line $\bR\cdot\smallmat{v\\1}$ and second, its length
$\norm{w_k}$ captures (the appropriate scaling of the) quality of approximation. It turns out that the normalization factor of $q_k^{-1/(d-1)}$ is the correct one for an interesting limit law to occur. 

Finally, in the pro-finite coordinate $\hat{\bZ}^d$ we simply take the integral vector $\smallmat{\tb{p}_k\\q_k}$ as it is (embedded diagonally in the product). This coordinate captures congruence questions regarding the $k$'th best approximation.

The whereabouts of the triple $(\Lam_k, w_k, \smallmat{\tb{p}_k\\q_k})$
in the product space $X_{d-1}\times \bR^{d-1}\times \hat{\bZ}^d$ and the statistical properties of the sequence of triples as $k$ changes is
very interesting from the perspective of Diophantine approximation. 
In \cite{SWbest} the following theorem is proved:
\begin{theorem}[\cite{SWbest}]\label{thm:SWbest}
Let $d\ge 2$ and let $\norm{\cdot}$ be either the Euclidean norm or the sup norm on $\bR^{d-1}$. 
For an irrational vector $v\in\bR^{d-1}$ let $\smallmat{\tb{p}_k(v)\\ q_k(v)} = \smallmat{\tb{p}_k\\ q_k} $ be the sequence of best
approximations of $v$. Then 
\begin{enumerate}
\item 
There exists a probability measure  $\nu_{\tb{best}}\in \cP(X_{d-1}\times \bR^{d-1}\times \hat{\bZ}^d)$
such that for Lebesgue almost any $v$ 
$$\lim_K \frac{1}{K}\sum_{k=1}^K\del_{(\Lam_k,w_k, \smallmat{\tb{p}_k\\ q_k})} = \nu_{\tb{best}}.$$ 
\item\label{part2} If $\vec{\al} = (\al_1,\dots,\al_{d-1},1)^{\on{t}}$ is a vector whose coordinates span a totally real number field of degree $d$ over $\bQ$ then
there exists a probability measure  $\nu_{\tb{best}}^{\vec{\al}}\in \cP(X_{d-1}\times \bR^{d-1}\times \hat{\bZ}^d)$ such that 
$$\lim_K \frac{1}{K}\sum_{k=1}^K\del_{(\Lam_k,w_k, \smallmat{\tb{p}_k\\ q_k})} = \nu_{\tb{best}}^{\vec{\al}}.$$  
\end{enumerate}
\end{theorem}
The following theorem will be proved in \S\ref{sec:proofs-of-apps} as an application of Theorem~\ref{th110} together with a result from \cite{SWbest}.
\begin{theorem}\label{thm853}
Let $ (\al_1,\dots,\al_{d-1},1)^{\on{t}}$ be a vector whose coordinates span a totally real number field of degree $d$ over $\bQ$.
Let $\mb{i}_n = (i_{n,1},\dots,i_{n,d-1})\in \bZ^{d-1}$ be such that for any $j\ne r\in\set{1,\dots,d-1}$ both $i_{n,r}\to_n\infty$ and $\av{i_{n,r}-i_{n,j}}\to_n\infty$. 
Let $m\ne 0,\pm1$ be an
integer and  and let $$\vec{\al}_n \defi (m^{i_{n,1}}\al_1, m^{i_{n,2} }\al_2,\dots, m^{i_{n,d-1}}\al_{d-1},1).$$ 
Then, in the notation of Theorem~\ref{thm:SWbest}, 
$\nu_{\tb{best}}^{\vec{\al}_n}\slra{n}\nu_{\tb{best}}$.
\end{theorem}

\begin{remark}
Quite a bit is known about the measures $\nu_{\tb{best}}, \nu_{\tb{best}}^{\vec{\al}}$ appearing in Theorem~\ref{thm:SWbest} and in particular, about how different from each other they are. 
For example, it is known that 
\begin{enumerate}[(i)]
\item The projection of $\nubest$ to $X_{d-1}$ is (not equal to but) in the same measure class as the periodic $\SL_{d-1}(\bR)$-invariant measure, while
the projection of $\nuvbest$ is singular to it.
\item The projection of $\nubest$ to $\bR^{d-1}$ is a compactly supported absolutely continuous measure with respect to Lebesgue measure and if the norm
used to define the best approximations is the Euclidean one then it is $\SO_{d-1}(\bR)$-invariant. On the other hand $\nuvbest$ is singular to Lebesgue and 
is not $\SO_{d-1}(\bR)$-invariant.
\item The projection of $\nubest$ to $\hat{\bZ}^d$ is the normalized restriction of the Haar measure to the compact open set obtained as the closure
of $\bZ^d_{\on{prim}}$ embedded diagonally in $\hat{\bZ}^d$, while the projection of $\nuvbest$ is singular to it. 
\end{enumerate}
These differences highlight the significance of Theorem~\ref{thm853}.
\end{remark}

\subsection{Limiting distributions of periodic orbits.}\label{sec:ldto}
We now state our main  results
Theorems~\ref{th12},~\ref{th11} and fix the notation of the paper. Let $\mathbf G$ be a connected (with respect to the Zariski topology) algebraic group defined over $\mathbb Q$. Let $S$ be a finite set of valuations on $\mathbb Q$ which contains the Archimedean one, and denote by $S_f$ the subset of non-Archimedean places in $S$. Let $\mathbb Q_S=\prod_{p\in S}\mathbb Q_p$ and $\mathcal O_S=\mathbb Z[1/p:p\in S_f]$. Denote by $$\mathbf G(\mathbb Q_S)=\prod_{p\in S}\mathbf G(\mathbb Q_p),\quad\Gamma_S=\mathbf G(\mathcal O_S),\quad\Gamma_\infty=\mathbf G(\mathbb Z).$$ For an algebraic group $\mathbf F$, we write $\mathbf F^0$ for the connected component of $\mathbf F$ in the Zariski topology, and write $\mathbf F(\mathbb R)^0$ for the connected component of $\mathbf F(\mathbb R)$ as a Lie group. In this paper, we prove the following theorem.

\begin{theorem}\label{th12}
Let $\mathbf G$ be a connected reductive algebraic group defined over $\mathbb Q$ without nontrivial $\mathbb Q$-characters. Let $\mathbf H$ be a connected reductive $\mathbb Q$-subgroup of $\mathbf G$ without nontrivial $\mathbb Q$-characters, and suppose that $\mathbf H$ contains a maximal torus of $\mathbf G$. Let $H$ be a subgroup of finite index in $\mathbf H(\mathbb Q_S)$ and $\mu_{\Gamma_S\backslash\Gamma_SH}$ the natural invariant probability measure supported on $\Gamma_S\backslash\Gamma_SH$. Let $\{g_i\}_{i\in\mathbb N}$ be a sequence in $\mathbf G(\mathbb Q_S)$. Then the sequence $g_i^*\mu_{\Gamma_S\backslash\Gamma_SH}$ has a subsequence converging to an algebraic probability measure $\mu$ on $\Gamma_S\backslash\mathbf G(\mathbb Q_S)$.
\end{theorem}

As a corollary of Theorem \ref{th12}, we deduce the following theorem regarding the equidistribution of certain arithmetically related periodic orbits in a 
a real homogeneous space, which we will need to prove Theorem~\ref{th110}.

\begin{theorem}\label{th11}
Let $\mathbf G$ be a connected $\mathbb R$-split reductive $\mathbb Q$-group and $\mathbf T$ a maximal $\mathbb R$-split torus defined over $\mathbb Q$ in $\mathbf G$ which is $\mathbb Q$-anisotropic. Let $S$ be a finite set of valuations on $\mathbb Q$ containing the archimedean one, $\{g_i\}$ a sequence in $\mathbf G(\mathbb R)$ and $\{h_i\}$ a sequence in $\Gamma_S$. Then the set of weak$^*$ limits of the sequence $\mu_{\Gamma_\infty h_i\mathbf T(\mathbb R)^0g_i^{-1}}$ is not empty, and contains only algebraic probability measures on $\Gamma_\infty\backslash\mathbf G(\mathbb R)$.

Moreover, if for any non-central element $x\in\mathbf T(\mathbb Q)$, the sequence $(g_ixg_i^{-1},h_ixh_i^{-1})$ diverges in $\mathbf G(\mathbb Q_S)$, then $\mu_{\Gamma_\infty h_i\mathbf T(\mathbb R)^0g_i^{-1}}$ converges to some translate of $\mu_{\Gamma_\infty\backslash\mathbf G(\mathbb R)^0}$.
\end{theorem}
\begin{remark}
In the context of Theorem~\ref{th11}, for any $h\in\mathbf G(\mathbb Q)$ and any $g\in\mathbf G(\mathbb R)$, $\Gamma_\infty h\mathbf T(\mathbb R)^0g^{-1}$ is compact in $\Gamma_\infty\backslash\mathbf G(\mathbb R)$. The argument for the proof realizes these orbits as projections of certain deformations of a fixed periodic orbit in the $S$-arithmetic
extension of $\Gamma_\infty\backslash \mathbf G(\bR)$. In order to apply Theorem~\ref{th12} we need this new lifted orbit to be an orbit of a group containing a maximal torus. 
For this to hold, we need $\mathbf T$ to be $\bR$-split and this is the reason for the extra conditions appearing in the statement of Theorem~\ref{th11} on top of the ones appearing
in the statement of Theorem~\ref{th12}.
\end{remark}
\subsection{About the proof of the main Theorem~\ref{th12}}
We follow closely the structure of \cite{EMS96, EMS97} and divide the argument into
two parts:
\begin{enumerate}[(i)]
\item Establish non-escape of mass.
\item Establish algebraicity of limit measures.
\end{enumerate} 
We prove non-escape
of mass in \S\ref{nondiv} and then establish the algebraicity of the limit measures in \S\ref{mr}. 
In order to establish  
part (i) we go along the lines of \cite{EMS97} but prefer to use the more robust formalism of $(C,\alpha)$-good functions
a la Klainbock and Margulis \cite{KM98} which was developed in its $S$-arithmetic version in \cite{KT07}. 
To establish part (ii) we follow \cite{EMS96, GO11} and use linearization and
Ratner's classification theorem of measures invariant and ergodic under one parameter unipotent flows. Unlike the real
case in \cite{EMS96}, in which the key observation is that a limit measure is invariant under a one-parameter unipotent 
subgroup because it has an invariance group having a non-trivial nilpotent line in its Lie algebra, over the $p$-adics, this is not enough. This is due to the existence of compact unipotent groups. 
In \S\ref{gp} we deal with this issue and prove that the limit measures in our discussion are invariant under a full-one parameter unipotent flow where the challenge is to prove invariance under large unipotent elements rather than just small ones. This invariance is then used in \S\ref{mr} in conjuction with the $S$-arithmetic linearization technique and Ratner's measure classification theorem to establish (ii).
\begin{remark}
For more details about the S-arithmetic Ratner's theorem and linearization technique we refer the reader also to
\cite{MT94,R98,Tom00}. One of the possible extensions of the results in this paper is to establish a theorem about limiting distributions of deformations of adelic periodic orbits. For limiting distributions of adelic torus orbits, one can read e.g. \cite{ELMV09,ELMV,Kh19,V10,Z10,DS18,DS18b}.  In general, a sequence of deformations of a fixed periodic 
adelic torus orbit does not necessarily equidistribute in the ambient space, and one of the obstructions is the escape of mass phenomenon which leads to the limiting measures being not probability. One can refer to \cite[Example 2.11]{AS} for an example of this phenomenon for $\mathbf G=\PGL_2$ due to Ubis. It is interesting to know whether the escape of mass phenomenon disappears when $\mathbf G$ is of higher rank.
\end{remark}

\section{Growth rates of functions}\label{ca}
In this section, we study $(C,\alpha)$-good functions in local fields. The notion of $(C,\alpha)$-good functions is introduced in \cite{KM98}. One may read \cite{Kle10,EMS97} for more details. Here we use the definition in \cite{KT07} for local fields. The goal in this section is Proposition \ref{p31}, which shows that certain functions are $(C,\alpha)$-good. At the end of this section, we prove a weak analogue of the intermediate value theorem for $(C,\alpha)$-good functions in local fields (Proposition \ref{p33}) which will be important in the discussion in \S \ref{gp}.  

Let $k=\mathbb Q_p$ ($p=$ prime or $\infty$) with the $p$-adic valuation $|\cdot|_p$. Let $F$ be a finite extension of $k$. Then there exists a unique extension $|\cdot|_F$ of $|\cdot|_p$ on $F$ such that for any $x\in F$ $$|x|_F=|N_{F/k}(x)|_p^{\frac1{[F:k]}}$$ where $N_{F/k}(x)$ is the norm of $x$. Equipped with $|\cdot|_F$, $F$ is a complete field. Note that for any extensions $k\subset F\subset L$, we have $N_{F/k}\circ N_{L/F}=N_{L/k}$ and $[L:F][F:k]=[L:k].$ Hence one can take the limit over all the finite extensions of $k$ and get a well-defined norm $|\cdot|_{\bar k}$ on the algebraic closure $\overline k$ of $k$. In the following, we will use the notation $|\cdot|_F$ for the norm on any finite extension $F$ of $k=\mathbb Q_p$ obtained in this manner. We will write $\mathcal O_F:=\{x\in F:|x|_F\leq1\}.$

Let $F$ be a complete field with a norm $|\cdot|$, $X$ a metric space, $\mu$ a Borel measure on $X$ and $U$ a subset of $X$. Recall \cite{KT07} that a continuous function $f:U\to F$ is called $(C,\alpha)$-good on $U$ with respect to $\mu$ if for any open ball $B\subset U$ centered in $\supp\mu$, we have $$\mu(\{x\in B:|f(x)|<\epsilon\})\leq C\left(\frac\epsilon{\|f\|_{\mu,B}}\right)^\alpha\mu(B)$$ for any $\epsilon>0$. Here $$\|f\|_{\mu,B}=\sup\{|f(x)|:x\in B\cap\supp(\mu)\}.$$ In the sequel, we will study the case when $X$ is a product of local fields and the measure $\mu$ is the Haar measure on $X$.

\begin{lemma}[C.f.~{\cite[Lemma 2.4]{EMS97}}]\label{l31}
Let $F$ be a finite extension of $\mathbb Q_p$. Then there exists $\gamma\in\mathbb N$ such that the determinant of the matrix generated by $s_i^ke^{\lambda_j s_i}$ $(i=1,2,...,n^2, j=1,2,...,n, k=0,1,2,...,n-1)$ $$\left|\begin{array}{ccccccc}e^{\lambda_1s_1} & e^{\lambda_1s_2} & \cdots & e^{\lambda_1s_{n^2}} \\s_1e^{\lambda_1s_1} & s_2e^{\lambda_1s_2} & \cdots & s_{n^2}e^{\lambda_1s_{n^2}} \\\vdots & \vdots &\cdots & \vdots\\s_1^{n-1}e^{\lambda_1s_1} & s_2^{n-1}e^{\lambda_1s_2} & \cdots & s_{n^2}^{n-1}e^{\lambda_1s_{n^2}} \\e^{\lambda_2s_1} & e^{\lambda_2s_2} & \cdots & e^{\lambda_2s_{n^2}} \\ s_1e^{\lambda_2s_1} & s_2e^{\lambda_2s_2} & \cdots & s_{n^2}e^{\lambda_2s_{n^2}}\\ \vdots & \vdots & \vdots & \vdots \\s_1^{n-1}e^{\lambda_ns_1} & s_2^{n-1}e^{\lambda_ns_2} & \cdots & s_{n^2}^{n-1}e^{\lambda_ns_{n^2}}\end{array}\right|=\prod_{i<j}(\lambda_i-\lambda_j)^{m_{i,j}}\prod_{i<j}(s_i-s_j)\phi(\bm\lambda,\bm s)$$ where $m_{i,j}\in\mathbb N$, $\phi(\bm{\lambda},\bm{s})$ is an analytic function in $\bm{\lambda}=(\lambda_1,\lambda_2,\dots,\lambda_n)\in(p^\gamma\mathcal O_F)^n$ and $\bm s=(s_1,s_2,\dots,s_{n^2})\in(p^\gamma\mathcal O_F)^{n^2}$, and for any $\bm\lambda\in(p^\gamma\mathcal O_F)^n$, we have $\phi(\bm\lambda,\bm s)\neq0$.
\end{lemma}
\begin{proof}
This lemma is proved in \cite[Lemma 2.4]{EMS97} for $F=\mathbb C$ and $\mathbb R$. Here we give a proof for the $p$-adic case. Indeed, if we expand the determinant and the analytic function $\phi(\bm{\lambda},\bm s)$ as power series of $s_1,s_2,...,s_{n^2}$ and $\lambda_1,\lambda_2,...,\lambda_n$ over $\mathbb R$, the formula proved in \cite[Lemma 2.4]{EMS97} holds as an identity of power series with rational coefficients. By examining the ranges of convergence of power series $s_i^ke^{\lambda_j s_i}$ $(1\leq i\leq n^2,\;1\leq j\leq n,\;0\leq k\leq n-1)$ in the $p$-adic case, one concludes that the formula holds for $\lambda_1,\lambda_2,...,\lambda_n$ and $s_1,s_2,...,s_{n^2}$ in small neighborhoods of $0$ in $\mathcal O_F$. This completes the proof of the lemma.
\end{proof}

\begin{proposition}\label{p31}
Let $F$ be a finite field extension of $\mathbb Q_p$ and $n\in\mathbb N$.
\begin{enumerate}
\item If $p$ is Archimedean, then there exists $\delta_0>0$ such that for any $c_i\in\mathbb C$, any $\lambda_i\in\mathbb C$ $(1\leq i\leq n)$ and any interval $I$ in $\mathbb R$ of length at most $\delta_0>0$, the function $f:I\to\mathbb C$ defined by $$f(s)=\sum_{i=1}^n\sum_{l=0}^{n-1} c_{i,l}s^le^{\lambda_is},\quad s\in\mathbb Z_p,$$ is $(C,\alpha)$-good for some $C,\alpha>0$ depending only on $\delta_0>0$ and $n$. The measure on $I$ is chosen to be the Lebesgue measure.
\item If $p$ is non-Archimedean, then there exists $k\in\mathbb N$ such that for any $c_i\in F$ and any distinct $\lambda_i\in p^k\mathcal O_F$ $(1\leq i\leq n)$, the function $f:\mathbb Z_p\to F$ defined by $$f(s)=\sum_{i=1}^n\sum_{l=0}^{n-1} c_{i,l}s^le^{\lambda_is},\quad s\in\mathbb Z_p,$$ is $(C,\alpha)$-good for some $C,\alpha>0$ depending only on $F$ and $n$. Here we choose the measure $\mu$ on $\mathbb Z_p$ to be the Haar measure on $\mathbb Q_p$ with $\mu(\mathbb Z_p)=1$.
\end{enumerate}
\end{proposition}
\begin{proof}
The case for the Archimedean place $p=\infty$ is essentially proved in \cite[Corollary 2.10]{EMS97}. Here we give a proof for non-Archimedean places. Denote by $$d(s_1,s_2,\dots,s_{n^2}):=\left|\begin{array}{ccccccc}e^{\lambda_1s_1} & e^{\lambda_1s_2} & \cdots & e^{\lambda_1s_{n^2}} \\s_1e^{\lambda_1s_1} & s_2e^{\lambda_1s_2} & \cdots & s_{n^2}e^{\lambda_1s_{n^2}} \\\vdots & \vdots &\cdots & \vdots\\s_1^{n-1}e^{\lambda_1s_1} & s_2^{n-1}e^{\lambda_1s_2} & \cdots & s_{n^2}^{n-1}e^{\lambda_1s_{n^2}} \\e^{\lambda_2s_1} & e^{\lambda_2s_2} & \cdots & e^{\lambda_2s_{n^2}} \\ s_1e^{\lambda_2s_1} & s_2e^{\lambda_2s_2} & \cdots & s_{n^2}e^{\lambda_2s_{n^2}}\\ \vdots & \vdots & \vdots & \vdots \\s_1^{n-1}e^{\lambda_ns_1} & s_2^{n-1}e^{\lambda_ns_2} & \cdots & s_{n^2}^{n-1}e^{\lambda_ns_{n^2}}\end{array}\right|.$$ Let $B$ be an open ball in $\mathbb Z_p$ and define $$U_{B,f,\epsilon}:=\{s\in B :|f(s)|_F<\epsilon\}.$$  We may assume $\mu(U_{B,f,\epsilon})>0$. By computing the volume of $p$-adic open balls, one can inductively find $n^2$ points $s_1,s_2,\dots,s_{n^2}\in U_{B,f,\epsilon}$ such that for any $1\leq i\neq j\leq n^2$ $$|s_i-s_j|_p\geq\frac{\mu(U_{B,f,\epsilon})}{n^2}.$$ Now by interpolation formula, we have $$f(s)=\sum_{j=1}^{n^2}f(s_j)\frac{d(s_1,\dots,s_{j-1},s,s_{j+1}\dots,s_{n^2})}{d(s_1,s_2,\dots,s_{n^2})},$$ and by Lemma \ref{l31} $$f(s)=\sum_{j=1}^{n^2}f(s_j)\frac{\prod_{i\neq j}(s-s_i)\phi(\bm\lambda,s_1,\dots,s,\dots,s_{n^2})}{\prod_{i\neq j}(s_j-s_i)\phi(\bm\lambda,s_1,\dots,s_{n^2})}$$ for some analytic function $\phi$. Here $\phi(\bm\lambda,\bm s)\neq0$ for $\bm\lambda$ in a small neighborhood $(p^l\mathcal O_F)^n$ of $\bm 0$.  This implies that for any $s\in B$ 
\begin{align*}
|f(s)|_F&\leq\max_{1\leq j\leq n^2}\left\{\left|f(s_j)\frac{\prod_{i\neq j}(s-s_i)\phi(\bm\lambda,s_1,\dots,s,\dots,s_{n^2})}{\prod_{i\neq j}(s_j-s_i)\phi(\bm\lambda,s_1,\dots,s_{n^2})}\right|_F\right\}\\
&\leq C\epsilon\frac{\mu(B)^{n^2-1}}{(\mu(U_{B,f,\epsilon})/n^2)^{n^2-1}}
\end{align*}
for some constant $C>0$ depending only on $n$ and $\phi$. Hence $$\mu(U_{B,f,\epsilon})\leq n^2C^{\frac1{n^2-1}}\left(\frac\epsilon{\sup_{s\in B}|f(s)|_F}\right)^{\frac1{n^2-1}}\mu(B).$$ We finish the proof of the proposition by letting $\alpha=1/(n^2-1)$ and replacing $n^2C^{\frac1{n^2-1}}$ by $C$.
\end{proof}

We conclude this section by proving the following proposition, which is a weak analogue of the intermediate value theorem in the $p$-adic case.

\begin{proposition}[Weak intermediate value theorem]\label{p33}
Let $F$ be a finite extension of $\mathbb Q_p$ with norm $|\cdot|_F$ and $f:\mathbb Z_p\to F$ be a continuous $(C,\alpha)$-good function. Then there exists a constant $C'>0$ depending only on $C$, $\alpha$ and $p$ such that for any $n\in\mathbb N$ we have $$\sup_{s\in p^{n+1}\mathbb Z_p}|f(s)|_F\geq\sup_{s\in p^n\mathbb Z_p}|f(s)|_F/C'.$$ In particular, if $f(0)=0$ and $|f(u)|_F\geq\rho$ for some $u\in\mathbb Z_p$ and $\rho>0$, then one can find $v\in\mathbb Z_p$ such that $\rho/C'\leq|f(v)|_F<\rho.$
\end{proposition}
\begin{proof}
For any $n\in\mathbb N$, we denote by $S_n:=\sup_{s\in p^n\mathbb Z_p}|f(s)|_F.$ By the $(C,\alpha)$-good property of $f$, for any $\epsilon>0$
$$\mu(\{s\in p^n\mathbb Z_p:|f(s)|_F<\epsilon\})\leq C\left(\frac\epsilon{S_n}\right)^\alpha\mu(p^n\mathbb Z_p).$$ Let $$\epsilon_n=S_n\left(\frac{\mu(p^{n+1}\mathbb Z_p)}{2C\mu(p^n\mathbb Z_p)}\right)^\frac1\alpha=\frac{S_n}{(2Cp)^\frac1\alpha}$$ and we have $$\mu(\{s\in p^n\mathbb Z_p:|f(s)|_F<\epsilon_n\})\leq\frac12\mu(p^{n+1}\mathbb Z_p).$$ This implies that there exists $v\in p^{n+1}\mathbb Z_p$ such that $$|f(v)|_F\geq\epsilon_n=\frac{S_n}{(2Cp)^\frac1\alpha}.$$ We finish the proof of the first statement by taking $C'=(2Cp)^\frac1\alpha$.

For the second statement, by continuity of $f$, one can find $n\in\mathbb N$ such that $$\sup_{s\in p^n\mathbb Z_p}|f(s)|_p\geq\rho\text{ and }\sup_{s\in p^{n+1}\mathbb Z_p}|f(s)|_p<\rho.$$ Now we can apply the first statement and obtain $$\sup_{s\in p^{n+1}\mathbb Z_p}|f(s)|_F\geq\sup_{s\in p^n\mathbb Z_p}|f(s)|_F/C'\geq\rho/C'.$$ This implies that there exists $v\in p^{n+1}\mathbb Z_p$ such that $\rho/C'\leq|f(v)|_F<\rho.$
\end{proof}

As a corollary, we deduce the following result.
\begin{corollary}\label{c31}
Let $F$ be a finite extension of $\mathbb Q_p$ with norm $|\cdot|_F$. Let $C,\alpha>0$ and $\{f_i\}_{i\in\mathbb N}$ a sequence of continuous $(C,\alpha)$-good function from $\mathbb Z_p$ to $F$. Suppose $\sup_{s\in\mathbb Z_p}|f_i(s)|_F\to\infty$ as $i\to\infty.$ Then there is a sequence $\{s_i\}$ in $\mathbb Z_p$ such that $s_i\to0\text{ and } f_i(s_i)\to\infty.$
\end{corollary}
\begin{proof}
Let $C'$ be the constant as in Proposition \ref{p33}. Since $\lim_{i\to\infty}\sup_{s\in\mathbb Z_p}|f_i(s)|_F=\infty,$ one can find a sequence $\{k_i\}$ in $\mathbb N$ such that $$k_i\to\infty\text{ and }\lim_{i\to\infty}\frac1{C'^{k_i}}\sup_{s\in\mathbb Z_p}|f_i(s)|_F=\infty.$$ By Proposition \ref{p33}, for each $i\in\mathbb N$, we have $$\sup_{s\in p^{k_i}\mathbb Z_p}|f(s)|_F\geq\sup_{s\in\mathbb Z_p}|f(s)|_F/C'^{k_i}.$$ Hence, for each $i\in\mathbb N$, there is $s_i\in p^{k_i}\mathbb Z_p$ such that $$|f(s_i)|_F\geq\sup_{s\in\mathbb Z_p}|f(s)|_F/(2C'^{k_i}).$$ Then the sequence $\{s_i\}_{i\in\mathbb N}$ satisfies our requirement.
\end{proof}

\section{Nondivergence}\label{nondiv}
In this section, we study the non-escape of mass of $S$-arithmetic periodic orbits. The arguments in this section will rely on the results in \cite{KM98,KT07,TW03}. Since $\mathbf G$ is a reductive $\mathbb Q$-group without non-trivial $\mathbb Q$-characters, we can embed $\mathbf G$ into $\SL_n$ for some $n\in\mathbb N$. 

\subsection{$(C,\alpha)$-good maps and Nondivergence}\label{nondivergence}
Let $X$ be a metric space with a locally finite measure $\mu$ and $S$ a finite set of places of $\mathbb Q$. A map $$\phi:(X,\mu)\to\SL_n(\mathbb Q_S),\quad x\in X\mapsto\phi(x)=(\phi_p(x))_{p\in S}\in\mathbf{SL}_n(\mathbb Q_S),$$ where $\phi_p(x)\in\SL_n(\mathbb Q_p)$ ($p\in S$), is called $(C,\alpha)$-good if there exist constants $C$ and $\alpha>0$ such that for every place $p\in S$, each entry of $\phi_p$ is a $(C,\alpha)$-good function from $X$ to $\mathbb Q_p$.

For any $p$ in the finite set $S$ of places of $\mathbb Q$, let $\Lie(\SL_n(\mathbb Q_p))$ denote the Lie algebra of $\SL_n(\mathbb Q_p)$, and let $\exp_p$ be the associated exponential map which is defined in a neighborhood of identity in $\Lie(\SL_n(\mathbb Q_p))$. We write $\exp_S$ for the product map $$(\exp_p)_{p\in S}:\prod_{p\in S}\Lie(\SL_n(\mathbb Q_p))\to\prod_{p\in S}\SL_n(\mathbb Q_p)$$ and $\log_S$ for the inverse of $\exp_S$, whenever these maps are defined. For any set $W$ in $\prod_{p\in S}\SL_n(\mathbb Q_p)$, we denote by $\log_S(W)$ the image of $W$ under the map $\log_S$. We write $\Lie(F)$ for the Lie algebra of a ($p$-adic) Lie group $F$.

Let $p\in S$ and $\{X_1^{p},X_2^{p},...,X_l^{p}\}$ a basis of the Lie algebra $\Lie(\mathbf H(\mathbb Q_p))$ in a small neighborhood of $0$ where the exponential map $\exp_p$ is defined. By Jordan decomposition, we can write $X_i^p=Y_i^p+Z_i^p$ where $Y_i^p$ is semisimple, $Z_i^p$ is nilpotent and $Y_i^p$ and $Z_i^p$ commute. Then by Proposition \ref{p31}, the map $s\mapsto\exp_p(sX_i^p)=\exp_p(sY_i^p)\exp_p(sZ_i^p)$ $(s\in\mathbb Z_p)$ is $(C,\alpha)$-good for some $C$ and $\alpha>0$. Consequently, by \cite[Corollary 3.3]{KT07}, the map $$\Phi_p:(s_1,s_2,...,s_l)\mapsto\exp_p(s_1X_1^p)\exp_p(s_2X_2^p)\cdots\exp_p(s_lX_l^p)$$ is $(C,\alpha)$-good. Since the map $\Phi_p$ is a local diffeomorphism from a neighborhood of $0$ in $\Lie(\mathbf H(\mathbb Q_p))$ to $\mathbf H(\mathbb Q_p)$, the Lebesgue measure on a neighborhood of $0$ in $\Lie(\mathbf H(\mathbb Q_p))$ maps to a measure on $\mathbf H(\mathbb Q_p)$ which is absolutely continuous to the Haar measure on $\mathbf H(\mathbb Q_p)$.

\label{p41}\label{p42}

Now we introduce the notation in \cite{KT07}. Let $\mathbb Q_S^n$ be the $\mathbb Q_S$-module $\prod_{p\in S}\mathbb Q_p^n=\{(x_p)_{p\in S}: x_p\in\mathbb Q_p^n\}.$ Denote by $\mathfrak M(\mathbb Q_S,\mathcal O_S,n)$ the space of discrete $\mathcal O_S$-modules in $\mathbb Q_S^n$, which can be identified as $$\left\{g^{-1}\Delta: g\in\GL_n(\mathbb Q_S), \Delta\text{ is a $\mathcal O_S$-submodule in }\mathcal O_S^n\right\}.$$ Note that any discrete $\mathcal O_S$-module in $\mathbb Q_S^n$ is a free $\mathcal O_S$-module. We write $\Omega_{S,n}$ for the space of lattices in $\mathbb Q_S^n$ defined by $$\Omega_{S,n}:=\left\{g^{-1}\mathcal O_S^n: g\in\GL_n(\mathbb Q_S)\right\}\cong\GL_n(\mathcal O_S)\backslash\GL_n(\mathbb Q_S).$$

For any $v\in S$, the valuation $|\cdot|_v$ induces a norm $\|\cdot\|_v$ on $\mathbb Q_v^n$. The metric on $\mathbb Q_S^n$, which is induced by the metrics $\|\cdot\|_v$ $(v\in S)$ as a product metric, will be denoted by $\|\cdot\|$. The content $c(x)$ of $x=(x_v)_{v\in S}\in\mathbb Q_S^n=\prod_{p\in S}\mathbb Q_p^n$ is defined to be the product of $\|x_v\|_v$ $(v\in S)$, and will be used in \S\ref{vcs}.

Let $$\GL_n^1(\mathbb Q_S)=\{x=(x_p)_{p\in S}\in\GL_n(\mathbb Q_S):\prod_p|\det(x_p)|_p=1\}.$$ Define the space of unimodular lattices in $\mathbb Q_S^n$ by $$\Omega_{S,n}^1:=\left\{\Delta\in\Omega_{S,n}:\cov(\Delta)=1\right\}\cong\GL_n^1(\mathcal O_S)\backslash\GL_n^1(\mathbb Q_S).$$ Here we refer to \cite{KT07} for the definition of the covolume $\cov(\Delta)$ of a discrete $\mathcal O_S$-module $\Delta$ in $\mathbb Q_S^n$.

As $\mathbf G$ is embedded into $\SL_n$ for some $n\in\mathbb N$, we have $$X_S=\mathbf G(\mathcal O_S)\backslash\mathbf G(\mathbb Q_S)\subset\GL_n^1(\mathcal O_S)\backslash\GL_n^1(\mathbb Q_S)$$ and $X_S$ can be identified with a subspace of $\mathcal O_S$-modules of covolume one in $\mathbb Q_S^n$. Note that the embedding above is a proper map.

In the following, we prove the non-divergence of the orbits $\Gamma\backslash\Gamma Hg^{-1}$ for $g\in\mathbf G(\mathbb Q_S)$, where $H$ is a subgroup of finite index in $\mathbf H(\mathbb Q_S)$. We take a neighborhood of identity $U^*=\prod_{p\in S} U_p$ in $H$ ($U_p\subset\mathbf H(\mathbb Q_p)$) such that the exponential map is defined on $U_p$ ($\forall p\in S$) and $U_p$ is a subgroup of $\mathbf H(\mathbb Q_p)$ ($\forall p\in S_f$). For convenience, in the proof of Theorem \ref{th44}, we denote by $\mu$ the Haar measure on $\mathbb R^l\times\prod_{p\in S_f}\mathbb Q_p^l$ ($l=\dim\mathbf H$). The proof is similar to that of \cite[Theorem 3.4]{EMS97}.

\begin{theorem}\label{th44}
Let $\mathbf G$ be a reductive algebraic group defined over $\mathbb Q$, which does not admit nontrivial $\mathbb Q$-characters. Let $\mathbf H$ be a reductive $\mathbb Q$-subgroup of $\mathbf G$ admitting no nontrivial $\mathbb Q$-characters. Let $S$ be a finite set of places of $\mathbb Q$ containing the Archimedean place, $H$ a subgroup of finite index in $\mathbf H(\mathbb Q_S)$ and $g\in\mathbf G(\mathbb Q_S)$. Assume that there exists a compact subset $K_0$ in $X_S=\Gamma_S\backslash\mathbf G(\mathbb Q_S)$ such that for any $h\in\mathbf G(\mathbb Q_S)$, $$\Gamma_S\backslash\Gamma_S U^*h^{-1}\cap K_0\neq\emptyset.$$ Then for any $\epsilon>0$, there exists a compact subset $K\subset X_S$ such that $$\mu_{\Gamma_S\backslash\Gamma_S Hg^{-1}}(K)\geq1-\epsilon.$$
\end{theorem}
\begin{proof}
Let $\Omega$ be a fundamental domain of $\Gamma_S\backslash\Gamma_S H$ in $H$, $\mu_H$ the Haar measure on $H$, and let $\delta_0$ be as in Proposition \ref{p31}. Let $W$ be a compact subset in $\Omega$ such that 
\begin{align*}
\mu_H(W)/\mu_H(\Omega)\geq1-\epsilon
\end{align*}
We can construct a finite open cover $\mathcal U$ of $W$ in $H$, which consists of translates of $U^*$ $\{B_i\}_{i=1}^N=\{U^*h_i\}_{i=1}^N$, and here it is harmless to assume that the diameter of $U^*$ is equal to $\delta_0>0$. Note that the number $N$ depends only on $\Omega$ and $\epsilon>0$.
 
By the discussion in this section, for any $B_i$ ($1\leq i\leq N$), there is an open ball $U_i$ of radius at most $\delta_0>0$ in $\mathbb R^l$ such that one can find a $(C,\alpha)$-good map $$\phi_i:U_i\times\prod_{p\in S_f}\mathbb Z_p^l\to B_i$$ where $C$ and $\alpha$ are uniform for all $1\leq i\leq N$. Note that for each $i=1,2,\dots,N$, the map $\phi_i(s)^{-1}$ ($s\in U_i\times\prod_{p\in S_f}\mathbb Z_p^l$), which is the composition of $\phi_i$ with the inversion in the group $H$, is also $(C,\alpha)$-good. 

Now we prove the non-escape of mass of $\mu_{\Gamma\backslash\Gamma Hg^{-1}}$. We know that there is a compact subset $K_0$ in $\Gamma_S\backslash\mathbf G(\mathbb Q_S)$ such that $\Gamma_S\backslash\Gamma_SB_ig^{-1}$ intersects $K_0$ for every $1\leq i\leq N$. By \cite[Theorem 8.8, Lemma 9.2 and Lemma 8.6]{KT07}, this implies that there exists $\rho>0$ depending only on $K_0$, $n$ and $S$ such that for any $1\leq i\leq N$, and for any submodule $\Delta\subset\mathcal O_S^n$ $$\sup_{s\in U_{i}\times\prod_{p\in S_f}\mathbb Z_p^l}\cov(g\phi_{i}(s)^{-1}\Delta)\geq\rho.$$ Then we apply \cite[Theorem 9.3]{KT07} and for any $\epsilon>0$, there exists a compact subset $Q_\epsilon$ such that for each $i=1,2,\dots,N$ 
\begin{align*}
\mu(\{s\in U_i\times\prod_{p\in S_f}\mathbb Z_p^l:g\phi_i(s)^{-1}\mathcal O_S^n\notin Q_\epsilon\})\leq\frac\epsilon N\mu(U_i\times\prod_{p\in S_f}\mathbb Z_p^l).
\end{align*}
Now we compute
\begin{align*}
&\mu_{H}(\{s\in\Omega:\Gamma_S sg^{-1}\notin Q_\epsilon\})\\
\leq&\sum_{i=1}^N\mu_{H}(\{s\in B_i: \Gamma_Ssg^{-1}\notin  Q_\epsilon\})+\mu_H(\Omega-W)\\
\leq&C_1\sum_{i=1}^N\mu(\{s\in U_i\times\prod_{p\in S_f}\mathbb Z_p^l:g\phi_i(s)^{-1}\mathcal O_S^n\notin Q_\epsilon\})+\epsilon\\
\leq&\frac{C_1\epsilon}N\sum_{i=1}^N\mu\left(U_i\times\prod_{p\in S_f}\mathbb Z_p^l\right)+\epsilon\leq C\epsilon\mu_{H}(\Omega)+\epsilon=(C+1)\epsilon
\end{align*}
where $C_1$ and $C$ depend only on $\Omega$, $\epsilon>0$ and the local diffeomorphisms $\phi_i$ $(i=1,2,...,N)$. Hence we have $\mu_{\Gamma_S\backslash\Gamma_SHg^{-1}}(Q_\epsilon^c)\leq C\epsilon$, for some constant $C>0$. This completes the proof of the theorem.
\end{proof}

\subsection{Translates of a maximal torus visiting a compact subset}\label{vcs}
Let $\mathbf H\subset\mathbf G$ be connected reductive algebraic groups defined over $\mathbb Q$ without nontrivial $\mathbb Q$-characters and suppose that $\mathbf H$ contains a maximal torus of $\mathbf G$. Let $H$ be a subgroup of finite index in $\mathbf H(\mathbb Q_S)$. Let $U^*=\prod_{p\in S} U_p$ in $H$ ($U_p\subset\mathbf H(\mathbb Q_p)$) be defined as in \S\ref{nondivergence}, which is a neighborhood of identity in $H$ such that the exponential map is defined on $U_p$ ($\forall p\in S$) and $U_p$ is a subgroup of $\mathbf H(\mathbb Q_p)$ ($\forall p\in S_f$). In this subsection, we prove the following theorem.

\begin{theorem}\label{scmth41}
There exists a compact subset $K$ in $\Gamma_S\backslash\mathbf G(\mathbb Q_S)$ such that for any $g\in\mathbf G(\mathbb Q_S)$ $$\Gamma_S\backslash\Gamma_SU^*g\cap K\neq\emptyset.$$
\end{theorem}

We know that $\mathbf H$ contains a maximal torus in $\mathbf G$. It implies that $\mathbf H$ contains a maximal torus $\mathbf T$ of $\mathbf G$ which is defined and anisotropic over $\mathbb Q$ \cite{PR}. In the rest of this section, we will fix this maximal torus $\mathbf T$. Now we take a small neighborhood of identity $\Omega$ in $\mathbf T(\mathbb Q_S)$ such that $\Omega\subset U^*$. Then it suffices to prove the following

\begin{theorem}\label{scth41}
There exists a compact subset $K\subset\Gamma_S\backslash\mathbf G(\mathbb Q_S)$ such that for any $g$ in $\mathbf G(\mathbb Q_S)$ $$\Gamma_S\backslash\Gamma_S\Omega g\cap K\neq\emptyset.$$
\end{theorem}

In order to prove Theorem \ref{scth41}, we need the following proposition.

\begin{proposition}\label{scp41}
There exists a subset $Y_p$ of $\mathbf G(\mathbb Q_p)$ such that 
\begin{enumerate}
\item $\mathbf G(\mathbb Q_p)=Y_p\mathbf T(\mathbb Q_p)$
\item For any linear representation of $\mathbf G(\mathbb Q_p)$ on a vector space $V$ over $\mathbb Q_p$, there exists a constant $c>0$ such that $$\sup_{t\in\Omega}\|yt.v\|\geq c\|v\|,\quad\forall v\in V,\quad\forall y\in Y_p.$$
\end{enumerate}
\end{proposition}
\begin{proof}
The case $p=\infty$ is essentially proved in \cite[Proposition 4.4]{EMS97}. Here we give a proof for the $p$-adic case. Let $F$ be a Galois extension of $\mathbb Q_p$ such that $\mathbf T$ is $F$-split. By Iwasawa decomposition, there exists a compact subgroup $K_F$ of $\mathbf G(F)$ and a unipotent subgroup $U_F$ of $\mathbf G(F)$ such that $$G(F)=K_FU_F\mathbf T(F).$$ Let $\{\sigma_1,\sigma_2,...,\sigma_n\}=\textup{Gal}(F/\mathbb Q_p)$ and define $$S=\{x\in\mathbf T(F):\sigma_1(x)\sigma_2(x)...\sigma_n(x)=e\}.$$

We prove that $S\mathbf T(\mathbb Q_p)$ is a subgroup of finite index in $\mathbf T(F)$. We know that $\mathbf T(\mathbb Q_p)^n$ is a subgroup of finite index in $\mathbf T(\mathbb Q_p)$. Let $$E=\{x\in\mathbf T(F):\sigma_1(x)\sigma_2(x)\cdots\sigma_n(x)\in\mathbf T(\mathbb Q_p)^n\}.$$ Then $E$ is a subgroup of finite index in $\mathbf T(F)$. Now for any $x\in E$, there exists $w\in\mathbf T(\mathbb Q_p)$ such that $$\sigma_1(x)\sigma_2(x)\cdots\sigma_n(x)=w^n$$ and hence $w^{-1}x\in S$, $x\in S\mathbf T(\mathbb Q_p)$. This implies that $E=S\mathbf T(\mathbb Q_p)$ and $S\mathbf T(\mathbb Q_p)$ is a subgroup of finite index in $\mathbf T(F)$.

Let $\Lambda$ be a finite set of representatives of $\mathbf T(F)/(S\mathbf T(\mathbb Q_p))$ in $\mathbf T(F)$. Since $\Lambda$ normalizes $U_F$, we have $$\mathbf G(F)=K_FU_F\Lambda S\mathbf T(\mathbb Q_p)=(K_F\Lambda)U_F S\mathbf T(\mathbb Q_p)=KU_FS\mathbf T(\mathbb Q_p)$$ where $K$ is the compact subset $K_F\Lambda$ in $\mathbf G(F)$. Let $Y_p=(KU_FS)\cap\mathbf G(\mathbb Q_p)$. Then $\mathbf G(\mathbb Q_p)=Y_p\mathbf T(\mathbb Q_p)$.

We prove that the subset $Y_p$ satisfies the second claim in the proposition. Let $\rho$ be a representation of $\mathbf G(\mathbb Q_p)$ on a space $V$ over $\mathbb Q_p$, and we extend it to a presentation of $\mathbf G(F)$ over $F$. Let $y=kus\in Y_p$ for $k\in K$, $u\in U_F$ and $s\in S$. Then assuming Lemmas \ref{scl41} and \ref{scl42} below, we have for any $v\in V$
\begin{align*}
\sup_{t\in\Omega}\|yt.v\|&=\sup_{t\in\Omega}\|kust.v\|\geq c_1\sup_{t\in\Omega}\|ut.(s.v)\|\geq c_2\sup_{t\in\Omega}\|(t^{-1}ut).(s.v)\|\\
&\geq c_2c_3\|s.v\|\geq c_2c_3c_4\|v\|.
\end{align*}

Hence we only need to prove the following lemmas.
\begin{lemma}\label{scl41}
There exists a constant $c_3>0$ such that for any $v\in V$ $$\sup_{t\in\Omega}\|(t^{-1}ut).v\|\geq c_3\|v\|$$
\end{lemma}
\begin{proof}
Let $\Phi$ be the root system of $(\mathbf G(F),\mathbf T(F))$ and $\Phi_+$ the set of positive roots. Let $\alpha_1,\alpha_2,\dots,\alpha_n$ be the positive simple roots of $\mathbf T(F)$ in $\Phi_+$. Then there is an isomorphism $f_\Phi:\mathbf T(F)\to(F_\times)^n$ which maps $x\in\mathbf T(F)$ to $$f_\Phi(x):=(\alpha_1(x),\alpha_2(x),\dots,\alpha_n(x)).$$ Write $u=\exp(\overline u)$ where $$\overline u=\sum_{\alpha\in \Phi_+}u_\alpha\in\sum_{\alpha\in \Phi_+}\mathfrak g_\alpha.$$ Note that for $t\in\Omega$, $t^{-1}ut.v=\exp(\sum_{\alpha\in \Phi_+}\alpha(t)u_\alpha).$ Since the positive roots are generated by simple roots, it follows that  the set $\{t^{-1}ut:t\in\Omega\}$ is the image of the map $f$ defined by $$(\alpha_1(t),\alpha_2(t),\dots,\alpha_n(t))\mapsto\exp(\sum_{\alpha\in \Phi_+}\alpha(t)u_\alpha).$$ We conclude that $\{t^{-1}ut:t\in\Omega\}$ is the image of a polynomial map $f$ defined on $f_\Phi(\Omega)$ which maps $s=(s_1,s_2,\dots,s_n)\in f_\Phi(\Omega)$ to $$f(s_1,s_2,\dots,s_n):=\exp(\sum_{\alpha\in \Phi_+}\alpha(s_1,s_2,\dots,s_n)u_\alpha)$$ where $\alpha(s_1,s_2,\dots,s_n)$ is a polynomial in $s_1,s_2,...,s_n$ without constant term.

To finish the proof of the lemma, it is enough to prove that there exists a constant $c_3>0$ such that $$\sup_{s\in f_\Phi(\Omega)}\|f(s).v\|\geq c_3\|v\|.$$ Since $f$ is a polynomial map on each variable $s_i$ with a uniformly bounded degree, by \cite[Corollary 3.3]{KT07}, $f$ is $(C,\alpha)$-good for some $C$ and $\alpha$. We also have $f(0)=e$, so $f(0).v=v.$ Now choose a bounded region $D$ in $F^n$ containing $f_\Phi(\Omega)$, and by the $(C,\alpha)$-good property, there are constants $c_3>0$ depending only on $C$, $\alpha$ and $f_\Phi(\Omega)$ such that $$\mu_{F^n}(\{s\in D:\|f(s).v\|\leq c_3\|v\|\})\leq\frac12\mu_{F^n}(f_\Phi(\Omega))$$ where $\mu_{F^n}$ denotes the Haar measure on $F^n$. Hence there is $s\in f_\Phi(\Omega)$ such that $\|f(s).v\|\geq c_3\|v\|.$ This completes the proof of the lemma.
\end{proof}

\begin{lemma}\label{scl42}
There exists a constant $c_4>0$ such that for any $v\in V$ $$\|s.v\|\geq c_4\|v\|.$$
\end{lemma}
\begin{proof}
We write the weight space decomposition for $V_F$ over $F$ as $$V=\sum_{\alpha} V_\alpha$$ and for any $v\in V$ we write $v=\sum_{\alpha}v_\alpha$. Then $s.v=\sum_{\alpha}\alpha(s)v_\alpha$ and $$\sigma_i(s.v)=\sigma_i(s).v=\sum_{\alpha}\alpha(\sigma_i(s))v_\alpha.$$ We know that $\|\cdot\|$ and $\|\sigma_i(\cdot)\|$ are equivalent norms on the vector space $V_F$ over $F$, and for any $v\in V_F$, there exists a constant $\kappa>0$ such that $$\frac1\kappa\|v\|\leq\|\sigma_i(v)\|\leq\kappa\|v\|$$ $$\frac1\kappa\max_\alpha\|v_\alpha\|\leq\|v\|\leq\kappa\max_\alpha\|v_\alpha\|.$$ For any $v\in V$, let $v_\beta=\max\|v_\alpha\|$. Then $$\|s.v\|\geq\frac1\kappa\|\sigma_i(s).v\|\geq\frac1{\kappa^2}\|\beta(\sigma_i(s)).v_\beta\|$$ $$\|s.v\|^n\geq\frac1{\kappa^{2n}}\|\beta(\prod_{i=1}^n\sigma_i(s))\|\|v_\beta\|^n\geq\frac1{\kappa^{3n}}\|v\|^n.$$
This implies that $\|s.v\|\geq c_4\|v\|$ for some constant $c_4>0$.
\end{proof}
This completes the proof of the proposition.
\end{proof}

\begin{proof}[Proof of Theorem \ref{scth41}]
We embed $\mathbf G$ in $\SL_n$ for some $n$, and hence $$\mathbf G(\mathcal O_S)\backslash\mathbf G(\mathbb Q_S)\subset\SL_n(\mathcal O_S)\backslash\SL_n(\mathbb Q_S)\subset\GL_n^1(\mathcal O_S)\backslash\GL_n^1(\mathbb Q_S)$$ where $\GL_n^1(\mathcal O_S)\backslash\GL_n^1(\mathbb Q_S)$ is the space of unimodular lattices $\Omega_{S,n}^1$ in $\mathbb Q_S^n$. Note that the group action of $\mathbf G(\mathbb Q_S)$ on  $\GL_n^1(\mathcal O_S)\backslash\GL_n^1(\mathbb Q_S)$ is defined by $$\Delta\mapsto h^{-1}\Delta,\;\forall h\in\mathbf G(\mathbb Q_S),\;\forall\Delta\in\Omega_{S,n}^1.$$ 

For any $p\in S$, let $Y_p$ be the subset of $\mathbf G(\mathbb Q_p)$ in Proposition~\ref{scp41} and we have $\mathbf G(\mathbb Q_p)$=$\mathbf T(\mathbf Q_p)Y_p^{-1}$. For $g\in\mathbf G(\mathbb Q_S)$, we can write $g=(t_py_p^{-1})_{p\in S}$ where $t_p\in\mathbf T(\mathbb Q_p)$ and $y_p\in Y_p^{-1}$ and we have $$\Gamma_S\backslash\Gamma_S\Omega g=\Gamma_S\backslash\Gamma_S(t_p)_{p\in S}\Omega(y_p^{-1})_{p\in S}.$$ Since $\mathbf T$ is $\mathbb Q$-anisotropic, $\mathbf T(\mathcal O_S)\backslash\mathbf T(\mathbb Q_S)$ is compact and there exists $s\in\mathbf T(\mathbb Q_S)$ in a compact fundamental domain of $\mathbf T(\mathcal O_S)\backslash\mathbf T(\mathbb Q_S)$ such that $\Gamma_S(t_p)_{p\in S}=\Gamma_Ss.$ Let $y=(y_p)_{p\in S}.$ We have $$\Gamma_S\backslash\Gamma_S\Omega g=\Gamma_S\backslash\Gamma_Ss\Omega y^{-1}.$$

Let $\rho_k$ be the representation of $\mathbf G(\mathbb Q_S)$ on the $k$-th exterior product space $V_k$ of $\mathbb Q_S^n$ $(1\leq k\leq n)$. Then by Proposition~\ref{scp41}, for any $v\in V_k$ $$\sup_{t\in\Omega} c((yt)s.v)\geq C_0c(s.v)$$ for some constant $C_0>0$. Here $c(v)$ denotes the content of the element $v$. This implies that for every submodule $\Delta\subset\mathcal O_S^n$ $$\sup_{t\in\Omega}\cov((yt)s.\Delta)\geq C_0\cov(s.\Delta)\geq C_1\cov(\Delta)$$ for some constant $C_1>0$, as $s$ is in a fixed compact subset. Since $\Omega$ is an image of a $(C,\alpha)$-good map, by \cite[Lemma 8.2]{KT07}, we can apply \cite[Theorem 9.3]{KT07} and there is a compact subset $K$ in $\Omega_{S,n}^1$ depending only on $n$, $\Omega$, $S$, $C$ and $\alpha$ such that $yts\mathcal O_S^n\in K$ for some $t\in\Omega$. This implies that for any $g\in\mathbf G(\mathbb Q_S)$, $\Gamma_S\backslash\Gamma_S\Omega g\cap K\neq\emptyset$.
\end{proof}

We conclude this section by the following thoerem, which is a corollary of Theorems~\ref{th44} and \ref{scmth41}.

\begin{theorem}\label{scmth42}
Let $\mathbf G$ be a reductive algebraic group defined over $\mathbb Q$, which does not admit nontrivial $\mathbb Q$-characters. Let $\mathbf H$ be a reductive $\mathbb Q$-subgroup of $\mathbf G$ without nontrivial $\mathbb Q$-characters, which contains a maximal torus of $\mathbf G$. Let $S$ be a finite set of places of $\mathbb Q$ containing the Archimedean place, $H$ a subgroup of finite index in $\mathbf H(\mathbb Q_S)$ and $\{g_i\}\in\mathbf G(\mathbb Q_S)$. Then after passing to a subsequence, the sequence of probability measures $g_i^*\mu_{\Gamma\backslash\Gamma H}$ converges to a probability measure.
\end{theorem}

\section{Auxiliary results in group theory}\label{gp}
In this section, we prove some auxiliary lemmas for several arguments in \S \ref{mr}.

Recall that $H$ is a subgroup of finite index in $\mathbf H(\mathbb Q_S)$. For any $p\in S_f$, we denote by $H_p$ a small open compact subgroup of $\mathbf H(\mathbb Q_p)$ contained in $H$, where the exponential map $\exp_p$ is defined. The subgroup $H_p$ could be considered as a closed subgroup in $\SL_n(\mathbb Q_p)$ for some $n\in\mathbb N$. Note that all the eigenvalues of elements in $H_p$ stay in a small neighborhood of $1\in\mathbb Q_p$. In the following, for a sequence $\{W_i\}_{i\in\mathbb N}$ of subsets in $\SL_n(\mathbb Q_p)$, we say $\{W_i\}$ is unbounded if for any compact subset $K\subset \SL_n(\mathbb Q_p)$, there are infinitely many $W_i$ such that $W_i\not\subset K$. Otherwise, we say $\{W_i\}$ is bounded.

\begin{lemma}\label{l51}
Let $\{\gamma_i\}$ be a sequence in $\SL_n(\mathbb Q_p)$. Suppose $\{\Ad\gamma_i(H_p)\}$ ($i\in\mathbb N$) is unbounded. Then for any $N>0$, after passing to a subsequence, there exists a sequence $h_i\in H_p$ such that $$h_i\to e_p,\quad\Ad\gamma_i(h_i)\to u\quad (i\to\infty)$$  for some unipotent element $u$ in $\SL_n(\mathbb Q_p)$ with $|u|_p\geq N$.
\end{lemma}
\begin{proof}
Since $\{\Ad\gamma_i(H_p)\}$ is unbounded, by passing to a subsequence, there is $x_i=\exp(X_i)$ in $H_p$ with $X_i\in\Lie(\SL_n(\mathbb Q_p))$ such that $\Ad\gamma_i(x_i)\to\infty$. By Proposition \ref{p31} and \cite[Lemma 3.1]{KT07}, we know that the function $f_i:\mathbb Z_p\to\mathbb R$ defined by $$f_i(s)=\|\gamma_i(\exp(sX_i))\gamma_i^{-1}-e_p\|_p,\quad s\in\mathbb Z_p$$ is a $(C,\alpha)$-good function with some $C,\alpha>0$ uniformly for any $i\in\mathbb N$. Our choices of $x_i$'s then imply $f_i(1)\to\infty$ as $i\to\infty$. By Corollary \ref{c31} we can find $s_i\in\mathbb Z_p$ so that $$s_iX_i\to0,\quad f_i(s_i)\to\infty.$$ Therefore, by replacing $X_i$ with $s_iX_i$, we may assume $X_i\to0$. Now for any $N>0$, one can find $k\in\mathbb N$ such that $C'^k\geq N$ where $C'$ is the constant in Proposition \ref{p33}. Since $f_i(0)=0$ and $f_i(s_i)\to\infty$ as $i\to\infty$, we apply Proposition \ref{p33} with $\rho=C'^{k+1}$, and one can find $v_i\in\mathbb Z_p$ such that $$f_i(v_i)\in[C'^k,C'^{k+1}].$$ Let $h_i=\exp(v_iX_i)$. Then by passing to a subsequence, we have $$h_i\to e_p,\quad \Ad\gamma_i(h_i)\to u$$ for some $u\in\SL_n(\mathbb Q_p)$ with $|u-e_p|_p\in[C'^k,C'^{k+1}]$. Since $h_i\to e_p$, all the eigenvalues of $\Ad\gamma_i(h_i)$ converges to $1$. It follows that $u$ is a unipotent element with $|u|_p\geq N$. 
\end{proof}

\begin{lemma}\label{l52}
A closed subgroup of $\SL_n(\mathbb Q_p)$ containing arbitrarily large unipotent elements has a one-parameter unipotent subgroup.
\end{lemma}
\begin{proof}
The exponential map $\exp(X)$ can be defined for nilpotent elements $X$ in the Lie algebra of $\SL_n(\mathbb Q_p)$, and the $\log$ function $\log(u)$ can be defined for unipotent elements $u$ in $\SL_n(\mathbb Q_p)$. Moreover, the sets of nilpotent elements and of unipotent elements are both closed. Now let $\{u_i\}$ be a sequence of unipotent elements in a closed subgroup $M$ with $u_i\to\infty$. We take $X_i=\log u_i$ and we have $X_i\to\infty$. By passing to a subsequence, $|X_i|_pX_i$ converges to a nilpotent element $X$ with $|X|_p=1$. Now we prove that the one-parameter unipotent subgroup $$s\in\mathbb Q_p\mapsto\exp(sX)$$ is contained in $M$.

For $s\neq0$ and $i$ sufficiently large, since $\overline{\mathbb Z}=\mathbb Z_p$, one can find $m_i\in\mathbb Z$ such that $$\left|m_i-s|X_i|_p\right|_p\leq\frac1{i|X_i|_p}.$$ This implies that $$m_iX_i=\left(m_i-s|X_i|_p\right)X_i+s|X_i|_pX_i\to sX.$$ Since $\exp(m_iX_i)=u_i^{m_i}\in M$ and $M$ is closed, we have $\exp(sX)\in M$. This completes the proof of the lemma.
\end{proof}

The following is an immediate corollary of Lemmas \ref{l51} and \ref{l52}.
\begin{proposition}\label{p51}
Let $\{\mu_i\}$ be a sequence of probability measures on $\Gamma\backslash\mathbf G(\mathbb Q_S)$ which weakly converges to a measure $\mu$, and $p$ a non-Archimedean place in $S$. Suppose that each measure $\mu_i$ is invariant under $\Ad(\gamma_i)H_p$ for some $\gamma_i\in\mathbf G(\mathbb Q_p)$, and $\{\Ad(\gamma_i)H_p\}$ is unbounded. Then $\mu$ is invariant under a one-parameter unipotent subgroup.
\end{proposition}

Now we consider the case when $\{\Ad(\gamma_i)H_p\}$ is bounded.

\begin{proposition}\label{p54}
Let $\gamma_i\in\mathbf G(\mathbb Q_p)$. The sequence $\{\Ad(\gamma_i)H_p\}$ is bounded if and only if $\gamma_i Z_{\mathbf G}(\mathbf H)(\mathbb Q_p)$ is bounded in $\mathbf G(\mathbb Q_p)/Z_{\mathbf G}(\mathbf H(\mathbb Q_p))$.
\end{proposition}
\begin{proof}
The set of semisimple elements in $H_p$ is Zariski dense in $\mathbf H(\mathbb Q_p)$, so one can pick a finite number of semisimple elements $\alpha_i\in H_p$ such that $$Z_{\mathbf G}(\alpha_1)\cap Z_{\mathbf G}(\alpha_2)\cap...\cap Z_{\mathbf G}(\alpha_n)=Z_{\mathbf G}(\mathbf H(\mathbb Q_p)).$$ Consider the action of $\mathbf G(\mathbb Q_p)$ on itself by conjugation. Under this action, the stabilizer of $\alpha_i$ in $\mathbf G(\mathbb Q_p)$ equals $Z_{\mathbf G}(\alpha_i)$. This induces a map $\beta_i:\mathbf G(\mathbb Q_p)/Z_{\mathbf G}(\alpha_i)\to\mathbf G(\mathbb Q_p)$ defined by $$\alpha(gZ_{\mathbf G}(\alpha_i))=g\alpha_i g^{-1}.$$ By \cite{BT65,B91} or \cite[Proposition 1.6]{MT94}, this map is a homeomorphism between $\mathbf G(\mathbb Q_p)/Z_{\mathbf G}(\alpha_i)$ and the conjugacy class of $\alpha_i$. 

Suppose $\{\Ad(\gamma_i)H_p\}$ is bounded. Then the sequence $\{\Ad(\gamma_i)\alpha_k\}$ is bounded $(\forall k=1,2,...,n)$. The homeomorphism $\beta_i$ then implies that  the sequence $\{\gamma_iZ_{\mathbf G}(\alpha_k)\}$ is also bounded in $\mathbf G(\mathbb Q_p)/Z_{\mathbf G}(\alpha_k)$. Consequently, this implies that $\gamma_iZ_{\mathbf G}(\mathbf H(\mathbb Q_p))$ is bounded in $\mathbf G(\mathbb Q_p)/Z_{\mathbf G}(\mathbf H(\mathbb Q_p))$. Conversely, if $\{\gamma_iZ_{\mathbf G}(\mathbf H(\mathbb Q_p))\}$ is bounded in $\mathbf G(\mathbb Q_p)/Z_{\mathbf G}(\mathbf H(\mathbb Q_p))$, then one can find a compact subset $K$ in $\mathbf G(\mathbb Q_p)$ such that $\gamma_i\in KZ_{\mathbf G}(\mathbf H(\mathbb Q_p)).$ Then $\Ad(\gamma_i)H_p$ is contained in the compact subset $KH_pK^{-1}$.
\end{proof}

The following is an analogue of the previous results over $\mathbb R$ which can be proved in a similar manner.
\begin{proposition}\label{p55}
Let $\{\gamma_i\}$ be a sequence in $\mathbf G(\mathbb R)$. Consider the operator $$\Ad\gamma_i:\Lie(Z_{\mathbf G}(\mathbf H)(\mathbb R))\to\Lie(\mathbf G(\mathbb R)).$$ Then we have
\begin{enumerate}
\item If the sequence $\{\|\Ad\gamma_i\|\}$ is unbounded, then by passing to a subsequence, there exists a one-parameter unipotent subgroup $\{u(t):t\in\mathbb R\}$ in $\mathbf G(\mathbb R)$ satisfying the following condition: for any $s\in\mathbb R$, one can find a sequence $w_i\in Z_{\mathbf G}(\mathbf H)(\mathbb R)^0$ such that $\gamma_iw_i\gamma_i^{-1}\to u(s).$
\item The sequence $\{\|\Ad(\gamma_i)\|\}$ is bounded if and only if $\gamma_iZ_{\mathbf G}(\mathbf H)(\mathbb R)$ is bounded in $\mathbf G(\mathbb R)/Z_{\mathbf G}(\mathbf H)(\mathbb R)$.
\end{enumerate}
Here $\|\cdot\|$ denotes the operator norm.
\end{proposition}

Similarly, we obtain the following
\begin{proposition}\label{p56}
Let $\{\mu_i\}$ be a sequence of probability measures on $\Gamma\backslash\mathbf G(\mathbb Q_S)$ which weakly converges to a measure $\mu$. Suppose that each measure $\mu_i$ is invariant under $\Ad(\gamma_i)Z_{\mathbf G}(\mathbf H)(\mathbb R)^0$ for some $\gamma_i\in\mathbf G(\mathbb R)$, and the norms of the operators $\Ad(\gamma_i):\Lie(\mathbf T(\mathbb R))\to\Lie(\mathbf G(\mathbb R))$ are unbounded. Then $\mu$ is invariant under a one-parameter unipotent subgroup.
\end{proposition}

\section{Measure rigidity: proof of Theorem \ref{th12}}\label{mr}
In this section, we prove Theorem \ref{th12} using S-arithmetic Ratner's theorem. We will mainly follow the frameworks and the presentations of \cite{EMS96,GO11}.

In order to prove Theorem \ref{th12}, we make the following reduction. We know that the sequence $g_i^*\mu_{\Gamma_S\backslash\Gamma_SH}$ converges to a probability measure $\mu$. Then there exist $a_i\in H$, $\gamma_i\in\Gamma_S$ and $g\in\mathbf G(\mathbb Q_S)$ such that $\pi_S(g)\in\supp(\mu)$ and $\gamma_ia_ig_i^{-1}\to g.$ We have $$\gamma_i^*\mu_{\Gamma_S\backslash\Gamma_SH}=(\gamma_ia_ig_i^{-1})^*(g_i)^*\mu_{\Gamma_S\backslash\Gamma_SH}\to g^*\mu.$$ Therefore, to prove that $\mu$ is an algebraic measure, it suffices to study the sequence of probability measures $\gamma_i^*\mu_{\Gamma_S\backslash\Gamma_SH}$ where $\{\gamma_i\}$ is a sequence in $\Gamma_S$. Now for each $\gamma_i\in\Gamma_S$, we denote by $\rho_i(x):=\gamma_ix\gamma_i^{-1}$ the conjugation of $\gamma_i$ in $\mathbf G$. Then $$\gamma_i^*\mu_{\Gamma_S\backslash\Gamma_SH}=\mu_{\Gamma_S\backslash\Gamma_S\rho_i(H)}$$ is the Haar measure on the finite-volume orbit $\Gamma_S\backslash\Gamma_S\rho_i(H)\subset\Gamma_S\backslash\mathbf G(\mathbb Q_S)$. Without loss of generality, we may assume that $\mu_{\Gamma_S\backslash\Gamma_S\rho_i(H)}$ converges to a probability measure $\mu$. 

Let $\mathbf T$ be a maximal torus in $\mathbf G$ which is defined and anisotropic over $\mathbb Q$, and contained in $\mathbf H$. Suppose that there are infinitely many $i\in\mathbb N$ such that $\gamma_i\mathbf H\gamma_i^{-1}$ is contained in a connected $\mathbb Q$-subgroup $\mathbf F$ of dimension smaller than $\dim\mathbf G$. By \cite[Lemma 5.1]{EMS97} and the condition that $\mathbf T\subset\mathbf F$, $\mathbf F$ is reductive and has no non-trivial $\mathbb Q$-characters. Since $$\mathbf T(\mathbb C)\subset\mathbf F(\mathbb C)\text{ and }\gamma_i\mathbf T(\mathbb C)\gamma_i^{-1}\subset\mathbf F(\mathbb C),$$ and both $\mathbf T(\mathbb C)$ and $\gamma_i\mathbf T(\mathbb C)\gamma_i^{-1}$ are maximal tori in $\mathbf F(\mathbb C)$, there exists an element $h_i\in\mathbf F(\mathbb C)$ such that $h_i\mathbf T(\mathbb C)h_i^{-1}=\gamma_i\mathbf T(\mathbb C)\gamma_i^{-1}$. Hence $h_i^{-1}\gamma_i$ is in the normalizer of $\mathbf T(\mathbb C)$. Let $W$ be a finite set of representatives of Weyl group in $\mathbf G(\mathbb C)$. Then $h_i^{-1}\gamma_i\in \mathbf T(\mathbb C)W$. This implies that there exists $w_i\in W$ such that $\gamma_iw_i^{-1}\in\mathbf F(\mathbb C).$ By passing to a subsequence, we may assume that $w_i^{-1}=w$ for all $i\in\mathbb N$. Now let $\delta_i=\gamma_i\gamma_1^{-1}$ and $\mathbf H'=\gamma_1\mathbf H\gamma_1^{-1}$. Then we have $$\delta_i=\gamma_iw(\gamma_1w)^{-1}\in\mathbf F(\mathbb C)\cap\mathbf G(\mathcal O_S)=\mathbf F(\mathcal O_S)$$ and $\{\delta_i\mathbf H'\delta_i^{-1}\}=\{\gamma_i\mathbf H\gamma_i^{-1}\}$ is contained in the $\mathbb Q$-reductive group $\mathbf F$. If we write the conjugation by $\delta_i$ as $\rho_i'(x):=\delta_ix\delta_i^{-1}$ and $H'=\gamma_1H\gamma_1^{-1}$, then $$\mu_{\Gamma_S\backslash\Gamma_S\rho_i(H)}=\mu_{\gamma_iH\gamma_i^{-1}}=\mu_{\delta_iH'\delta_i^{-1}}=\mu_{\Gamma_S\backslash\Gamma_S\rho_i'(H')}$$ is actually a sequence of probability measures on $\mathbf F(\mathcal O_S)\backslash\mathbf F(\mathbb Q_S)$. Therefore, without loss of generality, we may assume that $\mathbf G$ is the smallest $\mathbb Q$-subgroup containing the sequence $\{\gamma_i\mathbf H\gamma_i^{-1}\}$ and $\gamma_i\in\mathbf G(\mathcal O_S)$. 

Moreover, it is harmless to replace $\Gamma_S$ by a subgroup of finite index $\Gamma$ in $\Gamma_S$, and consider the homogeneous space $\Gamma\backslash\mathbf G(\mathbb Q_S)$. Let $Z(\mathbf G)$ be the center of $\mathbf G$. Then by passing to the quotient space $\mathbf G/Z(\mathbf G)$, we may further assume that $\mathbf G$ is semisimple. We will denote by $\pi_S:\mathbf G(\mathbb Q_S)\to\Gamma\backslash\mathbf G(\mathbb Q_S)$ be the natural quotient map.

\begin{theorem}\label{th61}
Let $\mathbf G$ be a connected semisimple algebraic group defined over $\mathbb Q$. Let $\mathbf H$ be a connected reductive $\mathbb Q$-subgroup without nontrivial $\mathbb Q$-characters, and assume that $\mathbf H$ contains a maximal torus in $\mathbf G$. Let $\rho_i:\mathbf H\to\mathbf G(i\in\mathbb N)$ be a sequence of $\mathbb Q$-homomorphisms defined by $\rho_i(x)=\gamma_ix\gamma_i^{-1}$ $(\gamma_i\in\Gamma)$ where $\Gamma$ is a subgroup of finite index in $\Gamma_S$, such that there exists no proper $\mathbb Q$-subgroup $\mathbf M$ containing infinitely many $\rho_i(\mathbf H)$. Let $H$ be a subgroup of finite index in $\mathbf H(\mathbb Q_S)$. Suppose that the sequence $\{\mu_{\Gamma\backslash\Gamma\rho_i(H)}\}$ converges to a probability measure $\mu$ as $i\to\infty$. Then $\mu$ is an algebraic measure on $\Gamma\backslash \mathbf G(\mathbb Q_S)$.
\end{theorem}

Theorem~\ref{th12} is then an immediate corollary of Theorem~\ref{scmth42} and Theorem~\ref{th61}. The rest of this section will be devoted to the proof of Theorem~\ref{th61}.
  
\subsection{Prerequisites}\label{mr2}
Let $\mathbf M$ be an algebraic vareity defined over $\mathbb Q$, and $S$ a finite set of places of $\mathbb Q$ containing the Archimedean place. The Zariski topology on $\mathbf M(\mathbb Q_S)=\prod_{p\in S}\mathbf M(\mathbb Q_p)$ is defined to be the product of the Zariski topologies on $\mathbf M(\mathbb Q_p)$ ($p\in S$). On the other hand, the topologies on the local fields $\mathbb Q_p$ ($p\in S$) induce a topology on $\mathbf M(\mathbb Q_S)=\prod_{p\in S}\mathbf M(\mathbb Q_p)$, which we will refer to as Hausdorff topology on $\mathbf M(\mathbb Q_S)$.
 
In the following, a map $f:\mathbf M(\mathbb Q_S)\to\mathbb Q_S$ of the form $f=(f_p)_{p\in S},$ where $f_p$ is a $\mathbb Q_p$-valued function on $\mathbf M(\mathbb Q_p)$, will be called a $\mathbb Q_S$-valued function on $\mathbf M(\mathbb Q_S)=\prod_{p\in S}\mathbf M(\mathbb Q_p)$. The space of $\mathbb Q_S$-valued functions on $\mathbf M(\mathbb Q_S)$ is a $\mathbb Q_S$-module. A $\mathbb Q_S$-valued function on $\mathbf M(\mathbb Q_S)$ is called regular if $f_p$ is regular for every $p\in S$. 

The following lemma could be verified by \cite[Ch.1]{B91}.

\begin{lemma}\label{l62}
 Let $\rho_i:\mathbf H\to\mathbf G\;(i\in\mathbb N)$ be a sequence of $\mathbb Q$-homomorphisms defined by $\rho_i(x)=\gamma_ix\gamma_i^{-1}$, and $\mathbf V$ an affine variety on which $\mathbf G$ acts morphically. Then for any regular $\mathbb Q_S$-valued function $f$ on $\mathbf V(\mathbb Q_S)$, the maps $f(\rho_i|_{\mathbf H(\mathbb Q_S)}w)$ $(w\in\mathbf V(\mathbb Q_S),i\in\mathbb N)$ span a finitely generated $\mathbb Q_S$-module in the space of $\mathbb Q_S$-valued functions on $\mathbf H(\mathbb Q_S)$.
\end{lemma}

The following is an immediate corollary of \cite[Theorem 7.9]{PR94}. 

\begin{lemma}\label{c61}
Let $\mathbf H$ be as in Theorem \ref{th61}, and $S$ a finite set of places of $\mathbb Q$ containing the Archimedean place. Then there exists an open neighborhood of identity $W_{S_f}$ in $\mathbf H(\mathbb Q_{S_f})$ such that $\mathbf H(\mathbb Q)$ is dense in $\mathbf H(\mathbb R)\times W_{S_f}$.
\end{lemma}

\subsection{Invariant measures for unipotent flows and linearization}\label{mr3}
In this subsection, we briefly discuss invariant measures for unipotent flows and linearization in the $S$-arithmetic setting. One can refer to e.g. \cite{MT94,MT96,R95,R98,Tom00,GO11} for more details. Here we will follow the notation and the formulations in \cite{GO11}.

We say that a connected $\mathbb Q$-subgroup $\mathbf P$ of $\mathbf G$ is in class $\mathcal F$ ralative to $S$ if the radical of $\mathbf P$ is unipotent and every $\mathbb Q$-simple factor of $\mathbf P$ is $\mathbb Q_v$-isotropic for some $v\in S$. A closed subgroup $L$ of $\mathbf G(\mathbb Q_S)$ is in class $\mathcal H$ if there exists a connected $\mathbb Q$-subgroup $\mathbf P$ in class $\mathcal F$ relative to $S$ such that $L$ has a finite index in $\mathbf P(\mathbb Q_S)$ and $L_u$ acts ergodically on $\Gamma\backslash\Gamma L$ with respect to the $L$-invariant probability measure. Here $L_u$ is the subgroup in $L$ generated by unipotent subgroups.

For a closed subgroup $L$ of $\mathbf G(\mathbb Q_S)$, the Mumford-Tate subgroup, denoted by $MT(L)$, is defined to be the smallest connected $\mathbb Q$-subgroup $\mathbf G$ such that $$\bar L^0\subset\prod_{v\in S}MT(L)(\mathbb Q_v)$$ where $\bar L^0$ denotes the identity component of the Zariski closure of $L$ in $\mathbf G(\mathbb Q_S)$.

Let $X_S=\Gamma\backslash\mathbf G(\mathbb Q_S)$. Let $W$ be a closed subgroup of $\mathbf G(\mathbb Q_S)$ generated by one-parameter unipotent subgroups. For each $L\in\mathcal H$, define $$\mathcal N(L,W)=\{g\in\mathbf G(\mathbb Q_S):W\subset g^{-1}Lg\},\quad\mathcal S(L,W)=\bigcup_{M\in\mathcal H,MT(M)\subsetneq MT(L)}\mathcal N(M,W)$$ $$\mathcal T_L(W)=\pi_S(\mathcal N(L,W)-\mathcal S(L,W))$$ where $\pi_S:\mathbf G(\mathbb Q_S)\to\Gamma\backslash\mathbf G(\mathbb Q_S)$ denotes the canonical projection. Note that we have $$\mathcal N(L,W)=\{g\in\mathbf G(\mathbb Q_S):W\subset g^{-1}MT(L)(\mathbb Q_S)g\},$$ and for any $P,Q\in\mathcal H$ with $MT(P)=MT(Q)$, $$\mathcal N(P,W)=\mathcal N(Q,W),\quad\mathcal S(P,W)=\mathcal S(Q,W),\quad \mathcal T_P(W)=\mathcal T_Q(W).$$ By \cite[Lemma 6.10]{GO11}, for any $P,Q\in\mathcal H$, either $\mathcal T_P(W)\cap\mathcal T_Q(W)=\emptyset$ or $\mathcal T_P(W)=\mathcal T_Q(W)$.

Let $\mathcal F^*$ be the collection of $\Gamma$-conjugacy classes of Mumford-Tate subgroups of $L\in\mathcal H$. For each $[\mathbf L]\in\mathcal F^*$, choose one subgroup $L\in\mathcal H$ with $MT(L)=\mathbf L$, and the collection of such $L$'s will be denoted by $\mathcal H^*$. 

Let $L\in\mathcal H^*$. Let $\mathfrak g$ and $\mathfrak l$ denote the Lie algebras of $\mathbf G$ and $MT(L)$ respectively. For $d=\dim\mathfrak l$, consider the representation $$\wedge^d\Ad:\mathbf G\to\GL(\mathbf V_L)$$ where $\mathbf V_L=\wedge^d\mathfrak g$. Let $\mathbf V_L(\mathbb Q_S)=\prod_{p\in S}\mathbf V_L(\mathbb Q_p)$ and $p_L\in\wedge^d\mathfrak l(\mathbb Q)\in\mathbf V_L(\mathbb Q_S)$. Then we have a map $\eta_L:\mathbf G(\mathbb Q_S)\to\mathbf V_L(\mathbb Q_S)$ defined by $$\eta_L((g_p)_{p\in S})=(\wedge^d\Ad(g_p)p_L)_{p\in S}.$$ Let $\Gamma_L:=\{\gamma\in\mathbf G(\mathcal O_S): \gamma^{-1}MT(L)\gamma=MT(L)\}$ and denote by
\begin{align*}
\Gamma_L^0:=\{\gamma\in\mathbf G(\mathcal O_S):\eta_L(\gamma)=p_L\}=\{\gamma\in\Gamma_L:\det(\Ad(\gamma)|_\mathfrak l)=1\}.
\end{align*}
By definition, we have $\eta_L(\Gamma_L)\subset\mathcal O_S^\times\cdot p_L$ where $\mathcal O_S^\times$ denotes the group of units in $\mathcal O_S$.

Recall the notion of $S(v_0)$-small subsets of $\mathbf V_L(\mathbb Q_S)$. Fix $\delta>0$ such that for any $w\in S$, if $\alpha\in\mathcal O_S^\times$ satisfies $\max_{v\in S\setminus\{w\}}|1-\alpha|_v<\delta,$ then $\alpha=1$. Now let $v_0\in S$. A subset $C=\prod_{v\in S}C_v\subset \mathbf V_L(\mathbb Q_S)$ is $S(v_0)$-small if for any $v\in S\setminus\{v_0\}$ and $\alpha\in K_v^\times$, $\alpha C_v\cap C_v\neq\emptyset$ implies that $|1-\alpha|_v<\delta$. Note that for any $\alpha\in\mathcal O_S^\times$ and any $S(v_0)$-small subset $C$ of $\mathbf V_L(\mathbb Q_S)$, $\alpha C\cap C\neq\emptyset$ implies $\alpha=1.$ 

If $-p_L\notin\eta_L(\Gamma_L)$, we set $\overline{\mathbf V}_L(\mathbb Q_S):=\mathbf V_L(\mathbb Q_S);$ otherwise, set $\overline{\mathbf V}_L(\mathbb Q_S):=\mathbf V_L(\mathbb Q_S)/\{1,-1\}.$ Denote by $\bar\eta_L$ the composition map of $\eta_L$ with the quotient map $\mathbf V_L(\mathbb Q_S)\to\overline{\mathbf V}_L(\mathbb Q_S)$. Denote by $A_L$ the Zariski closure of $\bar\eta_L(\mathcal N(L,W))$ in $\overline{\mathbf V}_L(\mathbb Q_S)$. Then $\bar\eta_L^{-1}(A_L)=\mathcal N(L,W).$ For $D$ a compact $S(v_0)$-small subset of $A_L$ for some $v_0\in S$, we define $$\mathcal S(D)=\{g\in\bar\eta_L^{-1}(D):\gamma g\in\bar\eta_L^{-1}(D)\text{ for some }\gamma\in\Gamma-\Gamma_L\}.$$

We finish this subsection by proving the following lemma.

\begin{lemma}\label{l65}
Let $\mathbf M$ be a connected $\mathbb Q$-group in class $\mathcal F$, and $M\in\mathcal H$ a subgroup of finite index in $\mathbf M(\mathbb Q_S)$. Suppose that $\mathbf M(\mathbb R)^0$ is generated by one-parameter unipotent subgroups. Then $$\overline{\Gamma\backslash\Gamma(\mathbf M(\mathbb R)^0\times\{e_f\})}=\Gamma\backslash\Gamma M.$$
\end{lemma}
\begin{proof}
By \cite[Theorem 2]{Tom00}, we have $\overline{\Gamma\backslash\Gamma\mathbf M(\mathbb R)^0}=\Gamma\backslash\Gamma F$ for some $F\in\mathcal H$ and $\mathbf M(\mathbb R)^0\subset F.$ This implies that $MT(M)\subset MT(F).$ On the other hand, we also have $$\Gamma\backslash\Gamma F=\overline{\Gamma\backslash\Gamma\mathbf M(\mathbb R)^0}\subset \Gamma\backslash\Gamma M.$$ So by \cite[Lemma 6.7]{GO11}, $MT(F)\subset MT(M).$ This implies that $MT(L)=MT(M)=\mathbf M$. Hence by \cite[Lemma 6.7]{GO11}, $\Gamma\backslash\Gamma F=\Gamma\backslash\Gamma M.$ This completes the proof of the lemma.
\end{proof}

\subsection{An Auxiliary Proposition}\label{mr4}
Let $\mathbf H$ be as in Theorem \ref{th61}. Let $\rho_i:\mathbf H\to\mathbf G(i\in\mathbb N)$ be a sequence of $\mathbb Q$-homomorphisms defined by $\rho_i(x)=\gamma_ix\gamma_i^{-1}$ $(\gamma_i\in\Gamma)$ where $\Gamma$ is a subgroup of finite index in $\Gamma_S$, such that there exists no proper $\mathbb Q$-subgroup $\mathbf M$ containing infinitely many $\rho_i(\mathbf H).$ Now for the group $H$, suppose that the sequence of subsets $\{\rho_i(H)\}$ in $\prod_{p\in S_f}\mathbf G(\mathbb Q_p)$ is bounded, and the norms $\|\Ad\gamma_i\|$ of the operators $\Ad\gamma_i:\Lie(\mathbf T(\mathbb R))\to\Lie(\mathbf G(\mathbb R))$ are bounded. Then by Proposition \ref{p54} and Proposition \ref{p55}, for each $p\in S$, we can find a bounded sequence $\gamma_{i,p}\in\mathbf G(\mathbb Q_p)$ and a sequence $t_{i,p}\in Z_{\mathbf G}(\mathbf H)(\mathbb Q_p)$ such that $\gamma_i=\gamma_{i,p}t_{i,p}.$ Therefore, $$\pi_S(\rho_i(H))=\Gamma\backslash\Gamma H\gamma_i^{-1}=\Gamma\backslash\Gamma (t_{i,p}^{-1})_{p\in S}H(\gamma_{i,p}^{-1})_{p\in S}.$$ Since $\mu_{\Gamma\backslash\Gamma\rho_i(H)}$ converges to a probability measure, the sequence $\{\Gamma(t_{i,p})_{p\in S}\}$ is bounded in the space $\Gamma\backslash\Gamma Z_{\mathbf G}(\mathbf H)(\mathbb Q_S)$. Let $$\Gamma(t_{i,p}^{-1})_{p\in S}=\Gamma(a_{i,p})_{p\in S}$$ where $(a_{i,p})_{p\in S}$ is a bounded sequence in $Z_{\mathbf G}(\mathbf H)(\mathbb Q_S)$. Then $$\pi_S(\rho_i(H))=\Gamma\backslash\Gamma H(a_{i,p}\gamma_{i,p}^{-1}).$$ By passing to a subsequence, the limiting measure of $\mu_{\Gamma\backslash\Gamma\rho_i(H)}$ is a translate of the periodic measure $\mu_{\Gamma\backslash\Gamma H}$. Theorem \ref{th61} then follows in this case. 

Now we assume that either the sequence $\{\rho_i(H)\}$ is unbounded in $\prod_{p\in S_f}\mathbf G(\mathbb Q_p)$ or the norms $\|\Ad\gamma_i\|$ of the operators $\Ad\gamma_i:\Lie(\mathbf T(\mathbb R))\to\Lie(\mathbf G(\mathbb R))$ are unbounded. Then by Proposition \ref{p51} and Proposition \ref{p56}, after passing to a subsequence, the limiting probability measure $\mu$ of $\mu_{\Gamma\backslash\Gamma\rho_i(H)}$ is invariant under a one parameter unipotent subgroup $W$.

We continue the notation and the discussion in the previous subsection. According to \cite[Theorem 6.11]{GO11}, there is some subgroup $L\in\mathcal H^*$ such that $\mu(\mathcal T_L(W))>0$ and $\mu(\mathcal S(L,W))=0$. Then one can find a compact subset $C^*\subset \mathcal T_L(W)$ such that $\mu(C^*)=\alpha>0$. Since $\bar\eta_L$ never takes the zero vector in every place $p\in S$, we may take the compact subset $C^*$ so that the subset  $$C_L:=\bar\eta_L(C^*)\subset A_L$$ is contained in a subset $\prod_{v\in S}C_v$ of $\mathbf V_L(\mathbb Q_S)$ which is $S(v_0)$-small for every $v_0\in S$.

Let $\Omega$ be a fundamental domain of $\Gamma\backslash\Gamma H$ in $H$, and $\mu_H$ the Haar measure on $H$. Let $\Omega_c$ be a compact subset in $\Omega$ such that $$\mu_H(\Omega-\Omega_c)\leq\epsilon_0\mu_H(\Omega)$$ where $\epsilon_0<\alpha/2$. Let $\{B_k\}_{k=1}^N$ be a cover of $\Omega_c$ with open boxes of diameter $\delta_0$, where $\delta_0$ is as in Proposition \ref{p31}. Moreover, we assume that for any $B_k$ ($1\leq k\leq N$), one can find an element $g_k\in\mathbf G(\mathbb Q_S)$ and an open box $U_k$ of radius at most $\delta_0>0$ around $0$ in $\Lie(\mathbf H(\mathbb Q_S))=\prod_{p\in S}\Lie(\mathbf H(\mathbb Q_p))$ such that $\exp_S(U_k)=B_kg_k^{-1}.$ The orbit $\Gamma\backslash\Gamma\rho_i(H)$ is then equal to $$\Gamma\backslash\Gamma\rho_i(H)=\bigcup_{k=1}^N\Gamma\backslash\Gamma\rho_i(B_k)\cup\Gamma\backslash\Gamma\rho_i(\Omega-\Omega_c).$$ Note that $N$ depends only on $\epsilon_0$ and $\Omega_c$. By Proposition \ref{p31}, \cite[Corollary 3.3]{KT07} and the discussion in \S\ref{nondiv}, for each $\rho_i$ and $B_k$, one can find a bijective $(C,\alpha)$-good map $$\phi_i:U_k\to\rho_i(B_k)$$ where $C$ and $\alpha$ are uniform for all $i\in\mathbb N$ and $1\leq k\leq N$.
 
Write $A_L$ as $\prod_{p\in S} V(f^{(p)}_1,f^{(p)}_2,\dots,f^{(p)}_n)$, where $f^{(p)}_i$ $(1\leq i\leq n)$ are polynomials on $\mathbf V_L(\mathbb Q_p)$ $(p\in S)$ and $V(f_1^{(p)},f_2^{(p)},...,f_n^{(p)})$ is the zero set of $f_i^{(p)}$ $(i=1,2,...,n)$. We define $$W_p(R,\alpha)=\{x\in\mathbf V_L(\mathbb Q_p):\|x\|_p\leq R,\quad|g_i^{(p)}(x)|_p\leq\alpha\}.$$ Define $R_0:=\sup_{x\in C_L}d(x,0).$ Then $\prod_{p\in S} W_p(R_0,\alpha)$ is a compact neighborhood of $C_L$ in $\mathbf V_L(\mathbb Q_S)$. 

\begin{proposition}\label{p64}
Let $C_L=\bar\eta_L(C^*)\subset A_L$ and $\phi_i$ defined as above. Fix $B_k$ for some $1\leq k\leq N$ ($\phi_i(U_k)=\rho_i(B_k)$). Then there exists a closed subset $\mathcal S$ in $\pi_S(\mathcal S(L,W))$ satisfying the following property: for a given compact subset $\mathcal K$ in $X_S-\mathcal S$, one of the following holds:
\begin{enumerate}
\item For any $\epsilon>0$ there exists a neighborhood $\Psi$ of $C_L$ such that $$\mu(\{x\in U_k: \pi_S(\phi_i(x))\in\mathcal K\cap\pi_S(\overline\eta_L^{-1}(\Psi))\})\leq\epsilon\mu(U_k)$$ for infinitely many $i\in\mathbb N$.
\item There exists $R>0$ such that for any $\alpha>0$ and for all large enough $i\in\mathbb N$, one can find $w_i\in\Gamma\bar p_L$ with $$\rho_i (B_k)w_i\subset \prod_{p\in S} W_p(R,\alpha).$$
\end{enumerate}
\end{proposition}
\begin{proof}
Suppose that the second claim does not hold. Let $\epsilon>0$, and take $R>0$ sufficiently large such that $R_0/R\leq\epsilon$. Then we can find $\alpha>0$ such that for infinitely many $i\in\mathbb N$ $$\rho_i(B_k)w\not\subset \prod_{p\in S} W_p(R,\alpha)$$ for any $w\in\Gamma\bar p_L$. In the sequel, we fix such an $R$.

For any $p\in S$ define $$D(p)=\prod_{v\neq p} C_{v}\times W_p(R,\alpha).$$ Then $D(p)$ is a compact $S(p)$-small subset which contains $\prod_{v\in S} C_v$. Let $\mathcal S:=\bigcup_{p\in S}\mathcal S(D(p)).$ For each $p\in S$ we take a small neighborhood of $D(p)$ $$\Psi(p)=\prod_{v\in S}\Psi_v(p)$$ as in \cite[Proposition 6.16]{GO11}, and define $$\Psi=(\prod_{p\in S} W_p(R_0,\epsilon\alpha))\cap\bigcap_{p\in S} \Psi(p).$$ We show that the subset $\Psi$ satisfies the first claim in the proposition.

Fix a compact $\mathcal K$ in $X_S-\mathcal S$. For $p\in S$, denote by $$\mathcal F_p=\{w\in\Gamma\bar p_L:\rho_i(B_k)w\not\subset \prod_{v\in S,v\neq p}\mathbf V_L(\mathbb Q_v)\times W_p(R,\alpha)\}.$$ By our assumption, $\bigcup_{p\in S}\mathcal F_p=\Gamma\bar p_L$ and we have
\begin{align*}
&\mu(\{x\in U_k: \pi_S(\phi_i(x))\in\mathcal K\cap\pi_S(\overline\eta_L^{-1}(\Psi))\})\\
=&\mu(\{x\in U_k:\pi_S(\phi_i(x))\in\mathcal K,\; \phi_i(x)\Gamma\bar p_L\cap\Psi\neq\emptyset\})\\
\leq&\sum_{p\in S}\mu(\{x\in U_k:\pi_S(\phi_i(x))\in\mathcal K,\; \phi_i(x)\mathcal F_p\cap\Psi\neq\emptyset\}).
\end{align*}
In the sequal, we will estimate $$\mu(\{x\in U_k:\pi_S(\phi_i(x))\in\mathcal K,\; \phi_i(x)\mathcal F_p\cap\Psi\neq\emptyset\})$$ for each $p\in S$.

For convenience, for any $v\in S$, we write $\pi_v$ for the natural projection $\Lie(\mathbf H(\mathbb Q_S))\to\Lie(\mathbf H(\mathbb Q_v))$, and $\pi^v$ for the natural projection from $\Lie(\mathbf H(\mathbb Q_S))$ to $\Lie(\mathbf H(\mathbb Q_{S\setminus\{v\}}))$. We denote by $\mu_{v}$ the Haar measure on $\Lie(\mathbf H(\mathbb Q_v))$, and $\mu_S$ the Haar measure on $\Lie(\mathbf H(\mathbb Q_S))$.

Case 1: $p\in S_f$. Define $$\mathcal S_p=\{y\in\Lie(\mathbf H(\mathbb Q_{S\setminus\{p\}})): \exists x\in U_k\text{ such that }\pi_S(\phi_i(x))\in\mathcal K,\;\phi_i(x)\mathcal F_p\cap\Psi\neq\emptyset,\;\pi^p(x)=y\}.$$ For $y\in\mathcal S_p$ and $w\in\mathcal F_p$, let $$I_w(y)=\{s\in\Lie(\mathbf H(\mathbb Q_p)):\exists x\in U_k\text{ such that }\;\pi_S(\phi_i(x))\in\mathcal K,\; \phi_i(x)w\in\Psi, \pi^p(x)=y,\pi_p(x)=s\}.$$ For any $s\in I_w(y)$, we denote by $B_{s,w}(y)$ the maximal ball in $U_k\cap\Lie(\mathbf H(\mathbb Q_p))$ containing $s$ such that 
\begin{enumerate}
\item $\phi_i(B_{s,w}(y))w\subset \overline{\Psi_p(p)}.$
\item There exists $s'\in B_{s,w}(y)$ with $\phi_i(s')w\in\overline{\Psi_p(p)}\setminus\Psi_p(p).$
\end{enumerate}
By definition of $\mathcal F_p$, such $B_{s,w}(y)$ exists and by \cite[Proposition 6.16]{GO11} $$\{B_{s,w}(y): w\in\mathcal F_p, s\in I_w(y)\}$$ is a collection of disjoint balls in $U_k\cap\Lie(\mathbf H(\mathbb Q_p))$, which we rewrite as $\{B_i(w,y)\}_{i\in\mathbb N}.$ Now since $\phi_i|_{U_k}$ is $(C,\alpha)$-good, for $y\in\mathcal S_p$ and $w\in\mathcal F_p$
\begin{align*}
&\mu_p(I_w(y))\\
\leq&\sum_{i=1}^\infty\mu_p(\{s'\in B_i(w,y): \exists x\in U_k\textup{ such that }\phi_i(x)w\in\Psi,\pi_p(x)=s',\pi^p(x)=y\})\\
\leq&\sum_{i=1}^\infty\epsilon\mu_p(B_i(w,y)).
\end{align*}
Then we have
\begin{align*}
&\mu_S(\{x\in U_k:\pi(\phi_i(x))\in\mathcal K,\; \phi_i(x)\mathcal F_p\cap\Psi\neq\emptyset\})\\
\leq&\sum_{w\in\mathcal F_p}\mu_S(\{x\in U_k:\pi(\phi_i(x))\in\mathcal K,\; \phi_i(x)w\in\Psi\})\\
=&\sum_{w\in\mathcal F_p}\int_{y\in\mathcal S_p}\mu_p(I_w(y))dy\leq\int_{y\in\mathcal S_p}\sum_{w\in\mathcal F_p}\sum_{i=1}^\infty\epsilon\mu_p(B_i(w,y))dy\\
\leq&\int_{y\in\mathcal S_p}\epsilon\mu_p(U_k\cap\Lie(\mathbf H(\mathbb Q_p)))dy\leq\epsilon\mu_S(U_k).
\end{align*}

Case 2: $p=\infty$. Define $$\mathcal S_\infty=\{y\in\Lie(\mathbf H(\mathbb Q_{S\setminus\{\infty\}})): \exists x\in U_k,\pi_S(\phi_i(x))\in\mathcal K,\;\phi_i(x)\mathcal F_\infty\cap\Psi\neq\emptyset\text{ such that }\pi^\infty(x)=y\}.$$ For $y\in\mathcal S_\infty$ let $$I(y)=\{s\in\Lie(\mathbf H(\mathbb R)): \exists x\in U_k,\;\pi(\phi_i(x))\in\mathcal K,\;\phi_i(x)\mathcal F_\infty\cap\Psi\neq\emptyset,\pi_\infty(x)=s,\pi^\infty(x)=y\}.$$ Now fix $y\in\mathcal S_\infty$ and choose $s_0\in I(y)$. By \cite[Proposition 6.16]{GO11}, one can find $w_0\in\Gamma\bar p_L$ such that $$\pi_S(\phi_i(s_0,y))\in\mathcal K,\;\phi_i(s_0,y)w_0\in\Psi$$ and for other $w\in\Gamma\bar p_L- \{w_0\}$, we have $\phi_i(s_0,y)w\notin\Psi(\infty).$ We split $I(y)$ into two subsets $$I(y,\mathcal F_\infty-\{w_0\}):=\{s\in\Lie(\mathbf H(\mathbb R)):(s,y)\in U_k,\;\pi_S(\phi_i(s,y))\in\mathcal K,\; \phi_i(s,y)(\mathcal F_\infty-\{w_0\})\cap\Psi\neq\emptyset\},$$ and $$I(y,w_0):=\{s\in\Lie(\mathbf H(\mathbb R)):(s,y)\in U_k,\;\pi_S(\phi_i(s,y))\in\mathcal K,\; \phi_i(s,y)w_0\in\Psi\}.$$

We first consider $I(y,\mathcal F_\infty-\{w_0\})$. Let $\vec v\in\Lie(\mathbf H(\mathbb R))$ and $\|\vec v\|=1$. Define $$I_{s_0}(\vec v):=\{x\in[0,\delta_0]:s_0+x\vec v\in U_k\cap\Lie(\mathbf H(\mathbb R))\}.$$ For $w\in\Gamma\bar p_L-\{w_0\}$, define $$I_{s_0,w}(\vec v):=\{x\in I_{s_0}(\vec v):\pi_S(\phi_i(s_0+x\vec v,y))\in\mathcal K,\;\phi_i(s_0+x\vec v,y)w\in\Psi\}.$$ For any $x\in I_{s_0,w}(\vec v)$, we denote by $B_{x,w}(\vec v)$ the maximal interval in $I_{s_0}(\vec v)$ containing $x$ such that 
\begin{enumerate}
\item $\phi_i(s_0+B_{x,w}(\vec v)\vec v,y)w\subset \overline{\Psi(\infty)}.$
\item There exists $x'\in B_{x,w}(\vec v)$ with $\phi_i(s_0+x'\vec v,y)w\in\overline{\Psi(\infty)}\setminus\Psi(\infty).$ 
\end{enumerate}
By definition of $\mathcal F_\infty$ and our choice of $s_0$ and $w_0$, such an interval $B_{x,w}(\vec v)$ exists, and by \cite[Proposition 6.16]{GO11} $$\{B_{x,w}(\vec v): w\in\mathcal F_\infty-\{w_0\}, x\in I_{s_0,w}(\vec v)\}$$ is a collection of intervals which cover $I_{s_0}(\vec v)$ at most twice. We rewrite this collection as $\{B_i(w,\vec v)\}_{i\in\mathbb N}$. Now since $\phi_i|_{U_k}$ is $(C,\alpha)$-good, for any $y\in\mathcal S_\infty$ and $\vec v\in S_1$
\begin{align*}
\sum_{w\in\mathcal F_\infty-\{w_0\}}\mu(I_{s_0,w}(\vec v))\leq&\sum_{w\in\mathcal F_\infty-\{w\}}\sum_{i=1}^\infty\mu(\{x'\in B_i(w,\vec v): \phi_i(s_0+x'\vec v,y)w\in\Psi\})\\
\leq&\sum_{w\in\mathcal F_\infty-\{w\}}\sum_{i=1}^\infty\epsilon\mu(B_i(w,\vec v))\leq2\epsilon\mu(I_{s_0}(\vec v))\leq2\epsilon\delta_0.
\end{align*}
 
We have
\begin{align*}
\mu(I(y,\mathcal F-\{w_0\}))\leq&\sum_{w\in\mathcal F_\infty-\{w_0\}}\int_{S^1}\int_{r\in[0,\delta_0]}\chi_{I_{s_0,w}(\vec v)}(r)r^{l-1}drd\vec{v}\\
\leq&\int_{S^1}\sum_{w\in\mathcal F_\infty-\{w_0\}}\mu(I_{s_0,w}(\vec v))\delta_0^{l-1}d\vec{v}\leq2\epsilon\delta_0^l=2\epsilon\mu(U_k\cap\Lie(\mathbf H(\mathbb R)))
\end{align*}
where $l=\dim\mathbf H$.

For $I(y,w_0)$, by our assumption on $\mathcal F_\infty$, one can find $s_1\in U_k$ such that $\phi_i(s_1,y)w_0\bar p_L\notin\Psi(\infty).$ We can then repeat the arguments above by replacing $s_0$ and $\mathcal F_\infty-\{w_0\}$ with $s_1$ and $\{w_0\}$, and obtain $\mu(I(y,w_0))\leq2\epsilon\mu(B_k).$
Therefore, we have
\begin{align*}
&\mu(\{x\in U_k:\pi(\phi_i(x))\in\mathcal K,\; \phi_i(x)\mathcal F_\infty\cap\Psi\neq\emptyset\})\\
\leq&\mu(\{x\in U_k:\pi(\phi_i(x))\in\mathcal K,\; \phi_i(x)(\mathcal F_\infty-\{w_0\})\cap\Psi\neq\emptyset\})\\
&\quad+\mu(\{x\in U_k:\pi(\phi_i(x))\in\mathcal K,\; \phi_i(x)w_0\in\Psi\})\\
\leq&\int_{y\in\mathcal S_\infty}\mu_\infty(I(y,\mathcal F-\{w_0\}))dy+\int_{y\in\mathcal S_\infty}\mu_\infty(I(y,w_0))dy\\
\leq&\int_{y\in\mathcal S_\infty}4\epsilon\mu_\infty(U_k\cap\Lie(\mathbf H(\mathbb R)))dy\leq4\epsilon\mu_S(U_k).
\end{align*}
Combining Case 1 and Case 2, we have
\begin{align*}
&\mu(\{x\in U_k: \pi(\phi_i(x))\in\mathcal K\cap\pi(\overline\eta_L^{-1}(\Psi))\})\\
\leq&\sum_{p\in S}\mu(\{x\in U_k:\pi(\phi_i(x))\in\mathcal K,\; \phi_i(x)\mathcal F_p\cap\Psi\neq\emptyset\})\\
\leq&\sum_{p\in S_f}\epsilon\mu(U_k)+4\epsilon\mu(U_k)\leq4\epsilon|S|\mu(U_k).
\end{align*}
We complete the proof of the proposition by replacing $4\epsilon|S|$ with $\epsilon$.
\end{proof}
 
\subsection{Proof of Theorem \ref{th61}}\label{mr5}
We will follow the arguments in \cite[\S 4]{EMS96} and \cite{EMS98}. Let $\mu$ be the limiting measure of $\mu_{\Gamma\backslash\Gamma\rho_i(H)}$. We continue the assumption that either the sequence $\{\rho_i(H)\}$ is unbounded in $\prod_{p\in S_f}\mathbf G(\mathbb Q_p)$ or the norms of the operators $\Ad\gamma_i:\Lie(\mathbf T(\mathbb R))\to\Lie(\mathbf G(\mathbb R))$ are unbound. Then $\mu$ is invariant under a one-parameter unipotent subgroup $W$. Let $L$ be a subgroup in $\mathcal H^*$ such that $\mu(\mathcal T_L(W))>0$ and $\mu(\mathcal S(L,W))=0$. Then one can find a compact subset $C\subset \mathcal T_L(W)$ such that $\mu(C)=\alpha>0$, and $C_L:=\bar\eta_L(C)\subset A_L$ is contained in a subset $\prod_{v\in S}C_v$ of $\mathbf V_L(\mathbb Q_S)$ which is $S(v_0)$-small for every $v_0\in S$. 

Let $\mathcal K$ be a compact neighborhood of $C$. The positivity of the measure $\mu(C)$ implies that the first claim in Proposition \ref{p64} does not hold for some $B_k$ ($1\leq k\leq N$). In other words, there exist $B_k$, $\alpha_i\to0$ and $w_i\in\Gamma$ such that $$\rho_i(B_k)w_i\bar p_L\subset\Phi_i$$ where $\Phi_i=\prod_{p\in S} W_p(R,\alpha_i)$ is a decreasing sequence of relatively compact subsets in $\bar V_L$ with $\cap_{i=1}^\infty\Phi_i\subset A_L$. 

Let $\mathcal S$ be the collection of all $\mathbb Q_S$-valued functions on $B_k$ of the form $$\omega=(\omega_p)_{p\in S}\mapsto (f_p(\rho_i(\omega_p)wp_L))_{p\in S},\quad\omega\in B_k$$ where $w\in\Gamma, i\in\mathbb N$ and $f_p$ is a linear functional on $\mathbf V_L(\mathbb Q_p)$ for each $p\in S$. By Lemma \ref{l62}, $\mathcal S$ spans a finitely generated $\mathbb Q_S$-module in the space of functions on $B_k$. Therefore, by \cite[Lemma 4.1]{EMS97}, there exists a finite set $\Sigma\subset B_k$ such that for any $\mathbb Q_S$-valued function $\{\phi_p\}_{p\in S}\in\mathcal S$, if $\{\phi_p:p\in S\}$ vanishes on $\Sigma$, then $\{\phi_p\}_{p\in S}=0$ on $B_k$. Note that since $\phi_p$ $(p\in S)$ are regular and $B_k$ are Zariski dense in $\mathbf H(\mathbb Q_S)$, this indeed implies $\{\phi_p\}_{p\in S}=0$ on $\mathbf H(\mathbb Q_S)$. By Lemma \ref{c61}, $\mathbf H(\mathbb Q)\cap B_k$ is dense in $B_k$. So we may assume that $\Sigma\subset\mathbf H(\mathbb Q_S)\cap B_k$.

Now for any $s\in\Sigma$ and $i\in\mathbb N$, $\rho_i(s)w_ip_L\subset\Phi_1.$ One can find $k\in\mathbb N$ such that $$\{\rho_i(s):i\in\mathbb N, s\in\Sigma\}\subset\mathbf G(\frac1k\mathcal O_S).$$ Since $\mathbf G(\frac1k\mathcal O_S)w_ip_L$ is a discrete subset of $\bar{\mathbf V}_L(\mathbb Q_S)$ and $\Phi_1$ is bounded, by passing to a subsequence, we have $\rho_i(s)w_ip_L=\rho_1(s)w_1p_L$ for all $i\in\mathbb N$ and $s\in\Sigma$. Therefore, by our choice of $\Sigma$, for all $\omega\in\mathbf H(\mathbb Q_S)$, we have $\rho_i(\omega)w_ip_L=\rho_1(\omega)w_1p_L.$ Now set $\omega=e_S$, we have $w_i p_L=w_1 p_L$. Thus $$\rho_i(\omega)w_1p_L=\rho_1(\omega)w_1p_L$$ for all $i\in\mathbb N$ and $\omega\in\mathbf H(\mathbb Q_S)$. Consequently, $w_1^{-1}\rho_1(\omega)^{-1}\rho_i(\omega)w_1\in N(MT(L))$ for all $i\in\mathbb N$ and $\omega\in\mathbf H(\mathbb Q)$, where $N(MT(L))$ is the normalizer of $MT(L)$ in $\mathbf G$. 

Without loss of generality, in the following, one may assume $w_1=e_S$. Therefore $$\rho_1(\omega)^{-1}\rho_i(\omega)\in N(MT(L)),\quad \omega\in\mathbf H(\mathbb Q).$$ We prove that $N(MT(L))$ equals $\mathbf G$. Let $\mathbf F$ be the smallest algebraic subgroup of $\mathbf G$ containing $\rho_1(\omega)^{-1}\rho_i(\omega)$ $(\omega\in\mathbf H(\mathbb Q)).$ Then $\mathbf F$ is defined over $\mathbb Q$. Let $\omega_1,\omega_2\in\mathbf H(\mathbb Q)$. We compute 
\begin{align*}
\rho_1(\omega_1)^{-1}(\rho_1(\omega_2)^{-1}\rho_i(\omega_2))\rho_1(\omega_1)=(\rho_1(\omega_2\omega_1)^{-1}\rho_i(\omega_2\omega_1))(\rho_1(\omega_1)^{-1}\rho_i(\omega_1))^{-1}\in\mathbf F(\mathbb Q).
\end{align*}
 This implies that $\rho_1(\mathbf H)\subset N(\mathbf F)$ where $N(\mathbf F)$ is the normalizer of $\mathbf F$, and $\rho_1(\mathbf H)\mathbf F$ is a $\mathbb Q$-subgroup. Since $\rho_i(\mathbf H)\subset \rho_1(\mathbf H)\mathbf F$ for all $i\in\mathbb N$ and $\mathbf G$ is the smallest $\mathbb Q$-group containing $\rho_i(\mathbf H)$ $(i\in\mathbb N)$, we have $\mathbf G=\rho_1(\mathbf H)\mathbf F.$ Hence $\mathbf F$ is normal in $\mathbf G$. 
 
By the definition of $A_L$, we know that $\rho_i(\omega)\in\mathcal N(L,W)\;(\omega\in\mathbf H(\mathbb Q_S))$. We have $$\mathbf G(\mathbb Q_S)=\mathbf F(\mathbb Q_S)\rho_1(\mathbf H(\mathbb Q_S))\subset\mathbf F(\mathbb Q_S)\mathcal N(L,W)=\mathcal N(L,W).$$ Let $P$ be the closed subgroup generated by $gWg^{-1}\;(g\in\mathbf G(\mathbb Q_S))$. Then we have $$\overline{\Gamma\backslash\Gamma P}=\Gamma\backslash\Gamma L.$$ For any $\omega\in\mathbf H(\mathbb Q)$, $\rho_1(\omega)$ normalizes $P$ and hence $$P\subset\rho_1(\omega)L\rho_1(\omega)^{-1},\quad\overline{\Gamma\backslash\Gamma P}\subset\Gamma\backslash\Gamma(\rho_1(\omega)L\rho_1(\omega)^{-1}).$$ This implies that $\rho_1(\omega)MT(L)\rho_1(\omega)^{-1}=MT(L)$ and hence $MT(L)$ is normal in $\mathbf G.$

Since $\mu(\pi_S(\mathcal S(L,W)))=0$, by \cite[Theorem 6.11]{GO11}, every $W$-ergodic component of $\mu$ is $L$-invariant, and hence so is $\mu$. Since $\mathbf G$ is $\mathbb Q$-semisimple, $MT(L)$ is an almost direct product of $\mathbb Q$-factors of $\mathbf G$. Since $L$ is a subgroup of finite index in $MT(L)(\mathbb Q_S)$, we may choose $L_0$ a subgroup of finite index in $MT(L)(\mathbb Q_S)$ such that $L_0$ is normal in $\mathbf G(\mathbb Q_S)$ and $L_0\subset L$. Hence $\mu$ is $L_0$-invariant. 

Now consider the quotient space $\overline X=(\Gamma/(\Gamma\cap L_0))\backslash(\mathbf G(\mathbb Q_S)/L_0)$, and denote by $\overline{\mu}_{\Gamma\backslash\Gamma\rho_i(H)}$ the pushforward of the measure $\mu_{\Gamma\backslash\Gamma\rho_i(H)}$ via the map $\Gamma\backslash\mathbf G(\mathbb Q_S)\to\overline X$. Since $\mathbf H$ is connected, we have $$\rho_i(H)/(L_0\cap\rho_i(H))\subset\rho_i(\mathbf H(\mathbb Q_S))/(\rho_i(\mathbf H(\mathbb Q_S))\cap L_0)\subset\mathbf G(\mathbb Q_S)/MT(L)(\mathbb Q_S)$$ and the measure $\overline{\mu}_{\Gamma\backslash\Gamma\rho_i(H)}$ is indeed supported on the space $(\Gamma/(\Gamma\cap L_0))\backslash(\mathbf G(\mathbb Q_S)/MT(L)(\mathbb Q_S))$. We complete the proof of Theorem \ref{th61} by induction on  $\dim\mathbf G$.

\section{Correspondence between $\Gamma_\infty\backslash\mathbf G(\mathbb R)$ and $\Gamma_S\backslash\mathbf G(\mathbb Q_S)$: Proof of Theorem \ref{th11}}\label{cor}
In this section, we prove Theorem~\ref{scmr} which is a sharpening of Theorem~\ref{th12} in the sense that it gives explicit algebraic conditions on the deforming sequence that ensures equidistribution in the ambient space. This is done in a special case  where $\mathbf G$ an $\mathbb R$-split semisimple group defined over $\mathbb Q$ and $\mathbf H$ a maximal $\mathbb R$-split and $\mathbb Q$-anisotropic torus in $\mathbf G$.
As a corollary of Theorem~\ref{scmr}, we will deduce Theorem~\ref{th11}.
\subsection{Theorem~\ref{th12} for an $\mathbb R$-split semisimple group $\mathbf G$ and $\mathbf H$ a maximal $\mathbb R$-split and $\mathbb Q$-anisotropic torus}

In the rest of this section, we will assume that $\mathbf G$ is an $\mathbb R$-split semisimple group defined over $\mathbb Q$ and $\mathbf T$ is a maximal $\mathbb R$-split and $\mathbb Q$-anisotropic torus in $\mathbf G$. 

\begin{lemma}\label{scl53}
Let $\mathbf G$ be an $\mathbb R$-split semisimple group and $\mathbf S$ a maximal $\mathbb R$-split and $\mathbb Q$-anisotropic torus in $\mathbf G$. Let $\mathbf F$ be a connected $\mathbb Q$-subgroup of $\mathbf G$ containing $\mathbf S$. Then $\mathbf F$ is reductive and has no non-trivial $\mathbb Q$-characters.
\end{lemma}
\begin{proof}
By \cite[Lemma 5.1]{EMS97}, $\mathbf F$ is reductive. Let $\chi$ be a $\mathbb Q$-character of $\mathbf F$. Then $\chi=1$ on $\mathbf S$, as $\mathbf S$ is $\mathbb Q$-anisotropic. Moreover, $\chi=1$ on any one parameter unipotent subgroup. Since $\mathbf S$ is maximal, $\mathbf F$ is generated by $\mathbf S$ and unipotent subgroups. Hence $\chi=1$ on $\mathbf F$. 
\end{proof}

\begin{lemma}\label{scl54}
Let $\mathbf G$ be an $\mathbb R$-split semisimple group and $\mathbf S$ a maximal $\mathbb R$-split and $\mathbb Q$-anisotropic torus in $\mathbf G$. Let $\mathbf F$ be a connected algebraic subgroup of $\mathbf G$ containing $\mathbf S$. Let $W$ be a set of representatives of the Weyl group of $\mathbf S(\mathbb R)$ in $\mathbf G(\mathbb R)$. Suppose that for some $\gamma\in\mathbf G(\mathbb R)$, $\gamma\mathbf S(\mathbb R)\gamma^{-1}\subset \mathbf F(\mathbb R).$ Then there is $w\in W$ such that $\gamma w\in\mathbf F(\mathbb R).$
\end{lemma}
\begin{proof}
By \cite[Lemma 5.1]{EMS97}, $\mathbf F$ is a reductive subgroup of $\mathbf G$. Now both $\mathbf S(\mathbb R)$ and $\gamma\mathbf S(\mathbb R)\gamma^{-1}$ are maximal $\mathbb R$-split tori in $\mathbf F(\mathbb R)$, so there is $x\in\mathbf F(\mathbb R)$ such that $\gamma\mathbf S(\mathbb R)\gamma^{-1}=x\mathbf S(\mathbb R)x^{-1}.$ This implies that $\gamma^{-1}x$ is in the normalizer of $\mathbf S(\mathbb R)$ in $\mathbf G(\mathbb R)$. Since the Weyl group of $\mathbf S(\mathbb R)$ is the normalizer of $\mathbf S(\mathbb R)$ modulo $\mathbf S(\mathbb R)$, one can find $w\in W$ and $y\in \mathbf S(\mathbb R)$ so that $\gamma^{-1}x=wy.$ Hence $\gamma w=xy^{-1}\in\mathbf F(\mathbb R).$ This completes the proof of the lemma.
\end{proof}

Let $\Lambda$ be a subgroup of finite index in $\mathbf T(\mathbb Z)\cap\mathbf T(\mathbb R)^0$. We think of $\Lambda$ and $\mathbf T(\mathbb Z)\cap\mathbf T(\mathbb R)^0$ as subgroups of $\prod_{p\in S_f}\mathbf G(\mathbb Q_p)$ by diagonal embedding. Since the closure of $\mathbf T(\mathbb Z)\cap\mathbf T(\mathbb R)^0$ is compact in $\prod_{p\in S_f}\mathbf G(\mathbb Q_p)$, we can choose $\Lambda$ so that it is contained in a small neighborhood of identity in $\prod_{p\in S_f}\mathbf G(\mathbb Q_p)$ where the exponential map of $p$-adic Lie groups are defined for $p\in S_f$. We denote by $F$ the closure of $\Lambda$ in $\prod_{p\in S_f}\mathbf G(\mathbb Q_p)$ and $H$ the group generated by $\Gamma_S\cap\mathbf T(\mathbb Q_S)$ and $\mathbf T(\mathbb R)\times F$. By strong approximation, $H$ is a subgroup of finite index in $\mathbf T(\mathbb Q_S)$ and $$\Gamma_S\backslash\Gamma_SH=\Gamma_S\backslash\Gamma_S(\mathbf T(\mathbb R)\times F).$$

By Theorem~\ref{th12}, one can conclude that for any $\{g_i\}\subset\mathbf G(\mathbb Q_S)$, the sequence $(g_i)^*\mu_{\Gamma_S\backslash\Gamma_SH}$ has a subsequence converging to an algebraic probability measure $\mu$ on $\Gamma_S\backslash\mathbf G(\mathbb Q_S)$. In order to prove Theorem~\ref{th11}, we have to review the arguments in \S\ref{mr} and analyze the limiting algebraic measure $\mu$.

\begin{lemma}\label{scl21}
The closure $\overline{\Gamma_S\backslash\Gamma_S\mathbf T(\mathbb R)^0}$ of $\Gamma_S\backslash\Gamma_S\mathbf T(\mathbb R)^0$ in $\Gamma_S\backslash\mathbf G(\mathbb Q_S)$ is equal to $\Gamma_S\backslash\Gamma_S(\mathbf T(\mathbb R)^0\times F).$
\end{lemma} 
\begin{proof}
By diagonal embedding, we know that $\Lambda$ is a lattice in $(\mathbf T(\mathbb R)^0\times F)$ and $\Gamma_S(\mathbf T(\mathbb R)^0\times F)$ is closed in $\Gamma_S\backslash\mathbf G(\mathbb Q_S)$. Now let $\Omega$ be a fundamental domain of $\Lambda\backslash\mathbf T(\mathbb R)^0$. Then for any $x\in\mathbf T(\mathbb R)^0$, one can find $x'\in\Omega$ and $\gamma\in\Lambda$ such that $x=\gamma^{-1} x'$, and $$\Gamma_S(x,e_f)=\Gamma_S(\gamma^{-1} x',e_f)=\Gamma_S\gamma(\gamma^{-1}x',e_f)=\Gamma_S(x',\gamma).$$
Note that $x'$ and $\gamma$ vary in $\mathbf T(\mathbb R)^0$ and $\Lambda$ respectively, and $\Lambda$ is dense in $F$. So we have $$\overline{\Gamma_S(\mathbf T(\mathbb R)^0,e_f)}=\Gamma_S(\mathbf T(\mathbb R)^0\times F).$$ This completes the proof of the lemma.
\end{proof} 

At the end of \S\ref{mr}, we prove that if $\mu$ is the limiting measure of $(g_i)^*\mu_{\Gamma_S\backslash\Gamma_SH}$, then $\mu$ is a translate of the $M$-invariant probability measure $\mu_{\Gamma_S\backslash\Gamma_S M}$ for some subgroup $M$ in $\mathbf G(\mathbb Q_S)$. Moreover, there is a $\mathbb Q$-subgroup $\mathbf M$ of $\mathbf G$ containing $\mathbf T$ which does not have nontrivial $\mathbb Q$-characters such that $M\subset\mathbf M (\bQ_S)$ is of finite index. Due to the condition that $\mathbf T$ is $\mathbb R$-split,  by Lemmas~\ref{l65} and \ref{scl21}, and reviewing the induction argument in \S\ref{mr}, we have $\mathbf M(\mathbb R)^0\subset M$ and $$\Gamma_S\backslash\Gamma_S M=\overline{\Gamma_S\backslash\Gamma_S\mathbf M(\mathbb R)^0}.$$

Furthermore, suppose that for any non-central element $x\in\mathbf T(\mathbb Q)$ the sequence $g_ixg_i^{-1}$ diverges in $\mathbf G(\mathbb Q_S)$. We prove that $\mathbf M=\mathbf G$. By the discussion in \S\ref{mr}, it is enough to prove that $\mathbf G$ is the smallest group containing the sequence $\gamma_i\mathbf T\gamma_i^{-1}$ ($i\in\mathbb N$) where $\gamma_i$ is defined at the beginning of \S\ref{mr}. Note that since for any non-central element $x\in\mathbf T(\mathbb Q)$ the sequence $g_ixg_i^{-1}$ diverges in $\mathbf G(\mathbb Q_S)$, the same is true for $\gamma_ix\gamma_i^{-1}$.

Suppose on the contrary that there is a subgroup $\mathbf H$ in $\mathbf G$ such that $\mathbf H$ contains infinitely many $\gamma_i\mathbf T\gamma_i^{-1}$. By \cite[Lemma 5.1]{EMS97} and Lemma \ref{scl53}, $\mathbf H$ is reductive and has no non-trivial $\mathbb Q$-characters. Since $$\mathbf T(\mathbb R)^0\subset\mathbf H(\mathbb R)\text{ and }\gamma_i\mathbf T(\mathbb R)^0\gamma_i^{-1}\subset\mathbf H(\mathbb R),$$ by Lemma \ref{scl54}, there exists a finite subset $W\subset\mathbf G(\mathbb R)$ such that for every $\gamma_i$, one can find $w_i\in W$ with $\gamma_iw_i\in\mathbf H(\mathbb R).$ By passing to a subsequence, we may assume that $w_i=w$ for all $i\in\mathbb N$. Now let $\delta_i=\gamma_i\gamma_1^{-1}$ and $\mathbf T'=\gamma_1\mathbf T\gamma_1^{-1}$. Then we have $$\delta_i=\gamma_iw(\gamma_1w)^{-1}\in\mathbf H(\mathbb R)\cap\mathbf G(\mathcal O_S)=\mathbf H(\mathcal O_S)$$ and $\{\delta_i\mathbf T'\delta_i^{-1}\}=\{\gamma_i\mathbf T\gamma_i^{-1}\}$ is contained in the $\mathbb Q$-reductive group $\mathbf H$. Since $\mathbf G$ is $\mathbb R$-split and $\mathbf T'$ is a maximal $\mathbb R$-split torus, there is a non-central element $x$ in $\mathbf T'$ such that $x$ is in the center of $\mathbf H$. This implies that $\delta_ix\delta_i^{-1}=x$ for any $i$, which contradicts the assumption that $\gamma_ix\gamma_i^{-1}$ diverges in $\mathbf G(\mathbb Q_S)$. 

Hence we proved the following
\begin{theorem}\label{scmr}
Let $\{g_i\}$ be a sequence in $\mathbf G(\mathbb Q_S)$. Let $\mu_{\Gamma_S(\mathbf T(\mathbb R)^0\times F)}$ be the periodic probability measure supported on $\Gamma_S\backslash\Gamma_S(\mathbf T(\mathbb R)^0\times F)$. Then the sequence $(g_i)^*\mu_{\Gamma_S(\mathbf T(\mathbb R)^0\times F)}$ has a subsequence converging to a periodic probability measure $\mu$ on $\Gamma_S\backslash\mathbf G(\mathbb Q_S)$. 

Moreover, assume that the measure $\mu$ is a translate of the $M$-invariant probability measure $\mu_{\Gamma_S\backslash\Gamma_S M}$ for some subgroup $M$ in $\mathbf G(\mathbb Q_S)$. Then there is a $\mathbb Q$-subgroup $\mathbf M$ of $\mathbf G$ containing $\mathbf T$ which does not have nontrivial $\mathbb Q$-characters such that $\mathbf M(\mathbb R)^0\subset M$, $M\subset\mathbf M (\bQ_S)$ is of finite index and $\Gamma_S\backslash\Gamma_S M=\overline{\Gamma_S\backslash\Gamma_S\mathbf M(\mathbb R)^0}.$ If for any non-central element $x\in\mathbf T(\mathbb Q)$ the sequence $g_ixg_i^{-1}$ diverges in $\mathbf G(\mathbb Q_S)$, then $\mathbf M=\mathbf G$.
\end{theorem} 

\subsection{Proof of Theorem~\ref{th11}}
Let $S$ be a finite set of places of $\mathbb Q$ containing the Archimedean one. We denote by $$K^S:=\prod_{p\notin S}\mathbf G(\mathbb Z_p),\; K_S:=\prod_{p\in S_f}\mathbf G(\mathbb Z_p)\textup{ and }K_f:=\prod_{p\neq\infty}\mathbf G(\mathbb Z_p).$$ We write $e_p$, $e_S$ and $e_f$ for the identity in $\mathbf G(\mathbb Q_p)$, $\prod_{p\in S}\mathbf G(\mathbb Q_p)$ and $\prod_{p\in S_f}\mathbf G(\mathbb Q_p)$ respectively. If necessary, for an $S$-arithmetic element $g\in\prod_{p\in S}\mathbf G(\mathbb Q_p)$, we will index $g$ to indicate its coordinates in the set of places $S$, that is, $g=(g_p)_{p\in S}$ where $g_p$ is the coordinate of $g$ at the place $p$.

It is well-known that $\Gamma_S\backslash\mathbf G(\mathbb Q_S)$ is a finite disjoint union of closed $\mathbf G(\mathbb R)K_S$-orbits $$\Gamma_S\backslash\mathbf G(\mathbb Q_S)=\bigcup_i\Gamma_S\backslash\Gamma_S x_i\mathbf G(\mathbb R)K_S$$ for some $x_1=e_S,x_2,\dots,x_n\in\mathbf G(\mathbb Q_S)$. One can define a natural projection map from the closed orbit $\Gamma_S\backslash\mathbf G(\mathbb R)K_S$ to $X_\infty=\Gamma_\infty\backslash G_\infty$ by sending $$\mathbf G(\mathcal O_S)g=\mathbf G(\mathcal O_S)(g_\infty,g_f)\in\mathbf G(\mathcal O_S)\backslash\mathbf G(\mathcal O_S)\mathbf G(\mathbb R)K_S$$ where $g_\infty\in\mathbf G(\mathbb R)$ and $g_f\in K_S$, to the point $\Gamma_\infty g_\infty$ in $\Gamma_\infty\backslash\mathbf G(\mathbb R)$. We denote it by $$\varpi_\infty(\Gamma_Sg)=\Gamma_\infty g_\infty.$$ Note that the map $\varpi_\infty$ is well-defined on $\Gamma_S\backslash\mathbf G(\mathbb R)K_S$, and intertwines with the actions of $\mathbf G(\mathbb R)$ on the spaces $\Gamma_S\backslash\mathbf G(\mathbb R)K_S$ and $\Gamma_\infty\backslash\mathbf G(\mathbb R)$. For a probability measure $\mu$ defined on $\Gamma_S\backslash\mathbf G(\mathbb R)K_S$, we denote by $\varpi_\infty^*(\mu)$ the pushforward of $\mu$ on $\Gamma_\infty\backslash G_\infty$.

Since $\mathbf T(\mathbb Z)$ is a subgroup of the compact group $\prod_{p\in S_f}\mathbf G(\mathbb Z_p)$, one can choose a subgroup $\Lambda$ of finite index in $\mathbf T(\mathbb Z)\cap\mathbf T(\mathbb R)^0$ such that $\Lambda$ is contained in a small neighborhood of identity in $\prod_{p\in S_f}\mathbf G(\mathbb Q_p)$ where the exponential map of $p$-adic Lie groups are defined for $p\in S_f$. Let $F$ be the closure of $\Lambda$ in $\prod_{p\in S_f}\mathbf G(\mathbb Q_p)$.

\begin{lemma}\label{l21}
For any $h\in\Gamma_S$ and $g\in\mathbf G(\mathbb R)$, $$\varpi_\infty(\Gamma_S(\mathbf T(\mathbb R)^0\times F)(g^{-1},h^{-1}))=\Gamma_\infty h\mathbf T(\mathbb R)^0g^{-1}.$$
\end{lemma} 
\begin{proof}
By Lemma \ref{scl21}, we know that $$\overline{
\Gamma_S\mathbf T(\mathbb R)^0}=\Gamma_S(\mathbf T(\mathbb R)^0\times F).$$
Now for any $g\in\mathbf G(\mathbb R)$ and $h\in\Gamma_S$
\begin{align*}
\Gamma_S(\mathbf T(\mathbb R)^0\times F)(g^{-1},h^{-1}))
=\overline{\Gamma_S(\mathbf T(\mathbb R)^0g^{-1}\times h^{-1})}=\overline{\Gamma_S(h\mathbf T(\mathbb R)^0g^{-1}\times e_f)}.
\end{align*}
By continuity of $\varpi_\infty$, we have
\begin{align*}
\varpi_\infty(\Gamma_S(\mathbf T(\mathbb R)^0\times F)(g^{-1},h^{-1}))=\varpi_\infty(\overline{\Gamma_S(h\mathbf T(\mathbb R)^0g^{-1}\times\{e_f\})}=\Gamma_\infty h\mathbf T(\mathbb R)^0g^{-1}.
\end{align*}
This completes the proof of the lemma.
\end{proof} 

\begin{lemma}\label{l22}
Let $\mu$ be a probability measure on $\Gamma_S\backslash\mathbf G(\mathbb Q_S)$ supported on $\Gamma_S\backslash\mathbf G(\mathbb R)K_S$. Suppose that for some subgroup $H_0$ in $\mathbf G(\mathbb R)$, $\mu$ is $H_0$-invariant and $\varpi_\infty^*(\mu)$ is supported on an $H_0$-orbit $xH_0$ in $\Gamma_\infty\backslash\mathbf G(\mathbb R)$. Then $\varpi_\infty^*(\mu)=\mu_{xH_0}$. In particular, for any $g\in\mathbf G(\mathbb R)$ and $h\in\Gamma_S$ $$\varpi_\infty^*((g,h)^*\mu_{\Gamma_S(\mathbf T(\mathbb R)^0\times F)})=\mu_{\Gamma_\infty h\mathbf T(\mathbb R)^0g^{-1}}.$$
\end{lemma}
\begin{proof}
Since $\mu$ is $H_0$-invariant, the probability measure $\varpi_\infty^*(\mu)$ is also $H_0$-invariant and hence equal to the Haar measure $\mu_{xH_0}$ on $xH_0$. The second claim follows from Lemma \ref{l21}.
\end{proof}

The next lemma is straightforward.
\begin{lemma}\label{l23}
Suppose that a sequence of measures $\mu_i$ on $\Gamma_S\backslash\mathbf G(\mathbb Q_S)$ converges weakly to $\nu$. Then $\varpi_\infty^*(\mu_i)$ converges weakly to $\varpi_\infty^*(\nu)$.
\end{lemma}

\begin{proof}[Deduction of Theorem~\ref{th11} from Theorems~\ref{th12} and \ref{scmr}]
By Theorems~\ref{th12} and \ref{scmr}, we know that after passing to a subsequence, $(g_i,h_i)^*\mu_{\Gamma_S(\mathbf T(\mathbb R)^0\times F)}$ converges to some algebraic measure $\mu$. Here the measure $\mu$ is supported on $\Gamma_S\backslash\Gamma_S Mg$ for some $g\in\mathbf G(\mathbb Q_S)$ and $M$ a subgroup of $\mathbf G(\mathbb Q_S)$. Moreover, there is a $\mathbb Q$-subgroup $\mathbf M$ of $\mathbf G$ which does not have nontrivial $\mathbb Q$-characters such that $\mathbf M(\mathbb R)^0\subset M$ and $\Gamma_S\backslash\Gamma_S M=\overline{\Gamma_S\backslash\Gamma_S\mathbf M(\mathbb R)^0}.$ Since 
\begin{align*}
\Gamma_S\backslash\Gamma_S(\mathbf T(\mathbb R)^0\times F)(g_i^{-1},h_i^{-1})=\overline{\Gamma_S\backslash\Gamma_S(h_i\mathbf T(\mathbb R)^0g_i^{-1}\times\{e_f\})}\subset \Gamma_S\backslash\Gamma_S\mathbf G(\mathbb R)K_S,
\end{align*}
the support $\Gamma_S\backslash\Gamma_S Mg$ of the limiting measure $\mu$ is also contained in $\Gamma_S\backslash\Gamma_S\mathbf G(\mathbb R)K_S$.
Hence there exist $\gamma\in\Gamma_S$ and $x=(x_\infty,x_f)\in\mathbf G(\mathbb R)\times\prod_{p\in S_f}\mathbf G(\mathbb Z_p)$ such that $g=\gamma x.$ We have
\begin{align*}
\Gamma_S\backslash\Gamma_SMg=&\overline{\Gamma_S(\mathbf M(\mathbb R)^0\times\{e_f\})(\gamma x)}=\overline{\Gamma_S(\gamma^{-1}\mathbf M(\mathbb R)^0\gamma\times\{e_f\})x} \end{align*}
and $$\varpi_\infty(\Gamma_S\backslash\Gamma_SMg)=\Gamma_\infty(\gamma^{-1}\mathbf M(\mathbb R)^0\gamma)x_\infty.$$
Since $\Gamma_\infty(\gamma^{-1}\mathbf M(\mathbb R)^0\gamma)x_\infty$ is closed and of finite volume, by Lemmas \ref{l22} and \ref{l23}, we conclude that $\mu_{\Gamma_\infty h_i\mathbf T(\mathbb R)^0g_i^{-1}}$ converges to the $x_\infty^{-1}(\gamma^{-1}\mathbf M(\mathbb R)^0\gamma)x_\infty$-invariant probability measure $\mu_{\Gamma_\infty(\gamma^{-1}\mathbf M(\mathbb R)^0\gamma)x_\infty}$. This finishes the proof of Theorem \ref{th11}. 
\end{proof}

\section{Proof of Theorems \ref{th110}, }\label{sec:proofs-of-apps}
We end the paper by showing how Theorem~\ref{th110} follows as an application of Theorem~\ref{th11} and then proving Theorem~\ref{thm853}.
\begin{proof}[Proof of Theorem \ref{th110}]
Let $m$, $\vec{\al}$, $\mb{i}_n$, $\vec{\al}_n$, $g_{\vec{\al}_n}$ and $s_n$ be as in the statement. 
We use the notation of \S\ref{sec:ldto}. Let $S_f$ be the set of primes dividing $m$ and $S = S_f\cup\set{\infty}$. Let $\mathbf G = \PGL_d$, $\Ga_\infty = \mathbf G(\bZ)$, 
and $X = \Ga_\infty\backslash \mathbf G(\bR)$. 
Let $A$ be the connected component of the identity (in the hausdorf topology) of the group of diagonal matrices in $\mathbf G(\bR)$. 
We need to prove that the sequence of periodic orbits 
$$ Y_n \defi \Ga_\infty g_{\vec{\al}_n}As_n^{-1}$$ 
equidistribute in $X$. The proof is an application
of Theorem~\ref{th11}. In order to apply Theorem~\ref{th11} we need to find a $\bQ$-anisotropic maximal torus defined over $\bQ$, $\mathbf T<\mathbf G$ and elements 
$h_n\in \Ga_S, g_n\in \mathbf G(\bR)$ such that $Y_n = \Ga_\infty h_n \mathbf T(\bR)^0 g_n$ and such that for any non-central $x\in \mathbf T(\bQ)$  the sequence 
$(g_nxg_n^{-1}, h_nxh_n^{-1})\in \mathbf G(\bQ_S)$ diverges. 

Since $\Ga_\infty g_{\vec{\al}}A $ is periodic, the intersection $\Ga_\infty\cap g_{\vec{\al}}A g_{\vec{\al}}^{-1}$ is a co-compact lattice in 
$g_{\vec{\al}}A g_{\vec{\al}}^{-1}$. We denote its Zariski closure by $\mathbf T<\mathbf G$. Thus, $\mathbf T$ is a $\bQ$-anisotropic $\bQ$-torus and 
$A = g_{\vec{\al}}^{-1} \mathbf T(\bR)^0 g_{\vec{\al}}$ and 
$$Y_n = \Ga_\infty g_{\vec{\al}_n} As_n^{-1} =\Ga_\infty \diag\pa{m^{i_1^{(n)}},\dots,m^{i_d^{(n)}}} g_{\vec{\al}}A s_n^{-1}= \Ga_\infty \diag\pa{m^{i_1^{(n)}},\dots,m^{i_d^{(n)}}}\mathbf T(\bR)^0 g_{\vec{\al}}s_n^{-1}.$$
So if we denote $h_n =\diag\pa{m^{i_1^{(n)}},\dots,m^{i_d^{(n)}}} \in \Ga_S$ and $g_n = g_{\vec{\al}}s_n^{-1}\in\mathbf G(\bR)$, then the equidistribution sought follows from Theorem~\ref{th11} 
once we show that for any  non-central $x\in \mathbf T(\bQ)$ we have that $h_n x h_n^{-1}\to\infty$ in $\prod_{p\in S_f} \mathbf G(\bQ_p)$. 
Let $x\in \mathbf T(\bQ)$ be a non-central element and by abuse of notation, identify $x$ with a rational matrix  in $\GL_d(\bQ)$. 
We may assume by passing to a subsequence if necessary that the order between the coordinates of the 
vector $\mb{i}$ is fixed. For simplicity we assume that for all $n$, $ i_1^{(n)}<i_2^{(n)}<\dots <i_d^{(n)}$. In this case
for $1\le j<r\le d$ such that $x_{jr}\ne 0$, the $(j,r)$'th coordinate of the conjugation $h_n x h_n^{-1}$ equals 
$m^{(i_j^{(n)} - i_r^{(n)})}x_{jr}$ which goes to $\infty$ in $\bQ_p$ for any $p\in S_f$ by our assumption that $\av{i_j^{(n)} - i_r^{(n)}}\to\infty$ and our assumption on the order. We conclude that the proof will be concluded once we show that $x$ cannot be lower triangular. 
If $x$ is lower triangular then since $g_{\vec{\al}}^{-1}\mathbf T(\bR) g_{\vec{\al}}$ is the group 
of diagonal matrices in $\mathbf G(\bR)$ we get a (rational) diagonal matrix $a$ such that 
$$xg_{\vec{\al}} = g_{\vec{\al}} a.$$
Recalling the definition of $g_{\vec{\al}}$ in \eqref{eq:2301} and comparing the first rows of the matrices in both sides of 
this equation we see (since $x$ is assumed to be lower triangular) that 
$$(x_{11}\sig_1(\al_1),x_{11}\sig_2(\al_1),\dots, x_{11}\sig_d(\al_1)) = (a_{11}\sig_1(\al_1),a_{22}\sig_2(\al_1),\dots, a_{dd}\sig_d(\al_1)).$$
As $\sig_j(\al_1)\ne 0$ we get that $a_{jj} =x_{11}$ which means that $a$ is a scalar matrix. This in turn implies that $x$ itself 
(which is conjugate to $a$) is a scalar matrix -- a contradiction to our assumption that $x$ is non-central. 
\end{proof}
\begin{proof}[Proof of Theorem \ref{thm853}]
Let 
$$\vec{\al}_n \defi (m^{i_{n,1}}\al_1, m^{i_{n,2} }\al_2,\dots, m^{i_{n,d-1}}\al_{d-1},1)$$
be as in the statement.  
Let $\bK\subset \bR$ be the totally real number field generated by the coordinates of the vectors $\vec{\al}_n$. Let $\vphi = (\sig_1,\dots, \sig_{d-1}, \on{id}_{\bK})$ be an ordering of the different embeddings of $\bK$ into $\bR$ so that the last embedding is chosen to be the identity. Using similar notation as in~\eqref{eq:2301}, consider the homothety class $x_{\vec{\al}_n}$ of the lattice in $\bR^d$ spanned by the rows of the matrix
$$g_{\vec{\al}_n} =
 \mat{- &m^{i_{n,1}}\vphi(\al_1)&-\\ &\vdots&\\ - &m^{i_{n,d-1}}\vphi(\al_{d-1})&- \\ - &\vphi(1)&-}.
$$
Theorem~\ref{th110} tells us that for any choice of real matrices $s_n$, the periodic orbits $s_nAx_{\vec{\al}_n}$ equidistribute in
the space of homothety classes of lattices in $\bR^d$ with respect to the uniform measure on that space. 
We can now apply \cite[Theorem 14.2]{SWbest}, which says that in this case, 
\begin{equation}\label{eq:com2}
\nu^{\vec{\al}_n}_{\tb{best}} \longrightarrow \nu_{\tb{best}}.
\end{equation}
In fact, \cite[Theorem 14.2]{SWbest} only needs as an input the fact that the orbits $s_nAx_{\vec{\al}_n}$ equidistribute for a specific sequence of invertible real matricces $s_n$ but since we know this for any such choice, there is no need to dwell on the particular formula of $s_n$ that appears there.
\end{proof}

\bibliographystyle{amsalpha}


\end{document}